\documentclass[10pt]{amsart}
\usepackage[normalem]{ulem}
\usepackage{a4wide}
\usepackage{amsmath,amsfonts,euscript,amssymb,amsthm}
\usepackage{verbatim}

\usepackage{amscd,graphicx,subfigure,color,float,caption,algorithm}
\usepackage{mathptmx,bm,bbm,dsfont}
\usepackage{leftidx,empheq}
\usepackage{multirow,hhline}
\usepackage{enumerate,enumitem}
\usepackage[parfill]{parskip}
\usepackage[bookmarks=true,bookmarksnumbered=true,bookmarkstype=toc,colorlinks, linkcolor={blue}, citecolor={blue}, urlcolor={blue}]{hyperref}
\allowdisplaybreaks

\newtheorem{definition}{Definition}[section] 
\newtheorem{theorem}{Theorem}[section]
\newtheorem{corollary}[theorem]{Corollary}

\newtheorem{lemma}[theorem]{Lemma}

\newcommand{\ie}{\setlength{\parskip}{0cm} \setlength{\itemsep}{0cm}}
\DeclareMathAlphabet{\mathcal}{OMS}{cmsy}{m}{n}
\newcommand{\la}{\left(}
\newcommand{\ra}{\right)}
\newcommand{\lb}{\left\langle}
\newcommand{\rb}{\right\rangle}
\newcommand\tsout{\bgroup\markoverwith{\textcolor{red}{\rule[0.5ex]{2pt}{1.4pt}}}\ULon}

\newcommand{\stkout}[1]{\ifmmode\text{\tsout{\ensuremath{#1}}}\else\tsout{#1}\fi}
\def\b{\mathbb}

\def\E{\mathbb E}
\def\P{\mathbb P}
\def\R{\mathbb R}
\def\S{\mathbb S}
\def\T{\mathbb T}

\def\L{\mathbb L}
\def\H{\mathbb H}
\def\dt{\ {\rm d}t}
\def\ds{\ {\rm d}s}
\def\dr{\ {\rm d}r}
\def\dW{\ {\rm d}W}
\def\dx{\ {\rm d}x}
\def\dy{\ {\rm d}y}
\def\d{{\rm d}}

\def\m{M}
\def\1{\mathbbm{1}}
\def\tr{\textcolor{red}}

\def\bal{\begin{equation*}\begin{aligned}} 
		\def\eal{\end{aligned}\end{equation*}}
\def\be{\begin{equation}\label} 
	\def\ee{\end{equation}}
\def\bd{\begin{definition}\label}
	\def\ed{\end{definition}}
\def\bt{\begin{theorem}\label}
	\def\et{\end{theorem}}
\def\bl{\begin{lemma}\label}
	\def\el{\end{lemma}}
\numberwithin{equation}{section}

\title{Stochastic Landau-Lifshitz-Gilbert equations for frustrated magnets under fluctuating currents}
\date{\today}

\begin{document}

\author{Beniamin Goldys}
\address{School of Mathematics and Statistics, The University of Sydney, Sydney 2006, Australia}
\email{beniamin.goldys@sydney.edu.au}
\author{Chunxi Jiao}
\address{Lehrstuhl für Angewandte Analysis, RWTH Aachen University, Aachen 52062, Germany}
\email{jiao@math1.rwth-aachen.de}
\author{Christof Melcher}
\address{Lehrstuhl für Angewandte Analysis, RWTH Aachen University, Aachen 52062, Germany}
\email{melcher@math1.rwth-aachen.de}
\thanks{This  work was supported by the Australian Research Council Project DP200101866 and the Deutsche Forschungsgemeinschaft (DFG, German Research Foundation) Project 442047500 - SFB 1481. CM is grateful for the support and hospitality of the Sydney Mathematical Research Institute (SMRI). 
} 

\maketitle

\begin{abstract}
	We examine a stochastic Landau-Lifshitz-Gilbert equation for a frustrated ferromagnet with competing first and second order exchange interactions exposed to deterministic and random spin transfer torques in form of transport noise. We prove existence and pathwise uniqueness of weak martingale solutions in the energy space. The result ensures the persistence of topological patterns, occurring in such magnetic systems, under the influence of a fluctuating spin current.
\end{abstract}

\tableofcontents

\section{Introduction}
Magnetic systems with higher order exchange interactions are known to host topological pattern in form of skyrmions and hopfions in dimension $d=2,3$, respectively, occurring as isolated solitons or condensed in a regular lattice \cite{Sutcliffe}. This can be seen as an emergent phenomenon arising from competing exchange interactions beyond nearest neighbours on the atomic scale, including the case of geometric frustration with alternating (anti-)ferromagnetic coupling \cite{Abanov}. The controlled manipulation and transport of such structures by means of external currents is at the core of possible applications in future information technologies and a key challenge for the mathematical theory \cite{Sampaio_etal}. A decisive aspect on the level of such nanostructure is stability with respect to fluctuations coming from external sources. 
The mathematical description of random effects in magnetism is generally based on stochastic Landau-Lifshitz-Gilbert equations (SLLG) for a governing micromagnetic energy and a noise. The existence of weak martingale solutions of the SLLG with heat-bath noise in dimension $ d \leq 3 $ and pathwise uniqueness in dimension $ d=1 $ have been studied in \cite{BrzezniakGoldysJegaraj_2013, BrzezniakGoldysJegaraj_2017, GoldysLeTran}.
	
The conventional mathematical theory of Landau-Lifshitz-Gilbert equations (LLG) is based on the Dirichlet energy arising from classical Heisenberg interaction of neighbouring spins. A crucial mathematical feature is the possibility of finite time singularities in spatial dimensions greater than one. Global regularity is only expected for small initial date, so that weak solution concepts become unavoidable. In the energy critical dimension $d=2$, the blow-up scenario of dissipative harmonic flows including LLG is well-understood \cite{GuoHong,Harpes,Struwe1985}. The bubbling analysis singles out a suitable notion of energy decreasing weak solutions, so-called Struwe solutions, that are unique in this class. Corresponding existence and uniqueness results in the stochastic cases are just starting to emerge \cite{Hocquet_sHMF}. The occurrence of topological singularities, however, goes hand in hand with the collapse of topological patterns. Bubbling analysis extend to models for chiral skyrmions \cite{DoringMelcher,KomineasMelcherVenakides_small}, which are lower order perturbations of the classical theory, while the prediction of global regularity versus finite time blow-up remains a major open problem.
	
Higher order bi-harmonic exchange interaction exclude topological singularities in spatial dimensions $d\le 3$ by means of an infinite energy barrier. Global existence of topological patterns coupled to a Vlasov-Maxwell equation for the electron distributions function of an interacting current has been confirmed in \cite{DoresicMelcher}. In this work, we develop a stochastic framework for a LLG system with higher order exchange interactions and fluctuating currents and provide conditions under which topological patterns will persists almost surely. The models under consideration are based on
interaction energies of the form
\[
E(M)= \frac{1}{2} \int_{\R^d} \la |\Delta \m|^2 + \lambda | \nabla \m|^2 + h |\m -e_3|^2 \ra \dx
\]
for magnetization fields $M:\R^d \to \S^2 \subset \R^3$ where $d=2,3$ with $\lambda \in \R$, $h \ge 0$ and $ e_3:=(0,0,1) $. 
By scaling we can assume $\lambda = \pm 1$. The case we are focusing on is $\lambda=-1$ corresponding to a frustrated magnet. 
In this case the usual stability condition for the ferromagnetic state $m = e_3$ is $h>1/4$ which 
implies coercivity $E(\m)\gtrsim \int |\nabla^2 \m|^2 + |\m - e_3|^2$. The unperturbed Landau-Lifshitz-Gilbert equation reads
\[
\partial_t \m = - \m \times \boldsymbol{H}(\m) -
\alpha \m \times \m \times \boldsymbol{H}(\m) 
\]
where $\alpha>0$ is the damping constant and 
\[
\boldsymbol{H}(\m)= - \left[ \Delta^2 \m + \Delta \m - h e_3 \right] 
\]
the effective field, i.e., minus the $L^2$ energy gradient. The dissipative nature implies a uniform bound of $\m-e_3$ in $H^2$ in terms of initial conditions. 
Here we are interested in the LLG dynamics driven by a fluctuating current. The mathematical description of such non-variational effects requires careful considerations about the specific formulation of dissipation in the style of Landau-Lifshitz or Gilbert, respectively. The so-called adiabatic spin transfer torque induced by a spin velocity field $v$ is introduced by adding the convection term $(v \cdot \nabla) \m$ to the micromagnetic torque $\m \times \boldsymbol{H}(\m)$ in LLG (note that this gives rise to two distinct terms). The phenomenological non-adiabatic spin transfer torque is induced by subtracting a perpendicular counterpart $\beta \m \times (v \cdot \nabla) \m$. The 
full torque becomes
\[
\boldsymbol{\tau}(\m) = \m \times \left[ \boldsymbol{H}(\m) - \beta (v \cdot \nabla) \m \right] + (v \cdot \nabla) \m
\]
so that 
\[
\partial_t \m = - \boldsymbol{\tau}(\m) - \alpha \m \times \boldsymbol{\tau}(\m).
\]
	
In the special case $\alpha=\beta$, referred to as the Galilean invariant case, the interaction reduces, in the above formulation of LLG, to a single term $(1+\alpha^2) (v \cdot \nabla)\m$. It is known that in general the non-Galilean model is required for an adequate description of current driven magnetic microstructures, e.g. domain walls and vortices \cite{KurzkeMelcherMoser2006, KurzkeMelcherMoser2011}.
We decompose the fluctuating spin current $v_{\rm fluc} = v + \xi $ into a deterministic temporally homogeneous part $v$ and a space-time noise $\xi = \dot{W}$ that is formally the generalised time derivative (in the framework of Stratonovich calculus) of a Wiener process $W$ with values in a vectorial Sobolev space to be specified below. We shall treat the deterministic part of the spin transfer torque in its general form including adiabatic and non-adiabatic terms without any balancing conditions. 
For the random part we focus on the Galilean invariant case, i.e., we drop stochastic terms of the form $\m \times (\xi \cdot \nabla) \m$ and leave only a linear transport term $ (\xi \cdot \nabla) \m $, which has been studied in models of stochastic Euler and Navier-Stokes equations \cite{BrzezniakFlandoliMaurelli_2016,MikuleviciusRozovskii}.
Therefore, the stochastic equation that we study in this paper is in the form 
\begin{equation}\label{eq: sLLG in intro}
\begin{aligned}
	\partial_t \m 
	&= -\m \times \boldsymbol{H}(\m) - \alpha \m \times \la \m \times \boldsymbol{H}(\m) \ra \\
	&\quad - (1+\alpha \beta) (v \cdot \nabla) \m - (\alpha - \beta) \m \times (v \cdot \nabla) \m - (\xi \cdot \nabla) M,  
\end{aligned}
\end{equation} 
where the noise is understood in the Stratonovich sense to ensure $ \m $ takes values in the unit sphere (norm constraint). 
This equation is written formally later in \eqref{eq: sLLG with Delta^2} along with a precise definition of $ W $.
	
For \eqref{eq: sLLG in intro}, the corresponding Stratonovich correction is in the form $ \frac{1}{2}(\xi \cdot \nabla)^2 \m $, which has a lower order than the leading bi-Laplacian term that arises from the dissipation 
\[
\alpha \m \times \la \m \times \Delta^2 \m \ra
= - \alpha \Delta^2 \m + \alpha \lb \m, \Delta^2 \m \rb \m 
\] 
under the norm constraint. 
Intuitively, in view of semigroup theory for mild solutions, this suggests that the noise does not need to be small or of a scale comparable to $ \alpha $ to guarantee existence of a solution. 
This is indeed the case, since in energy estimates, It{\^o}'s formula leads to cancellations between Stratonovich and It{\^o} corrections which reduce the order of norm required to estimate these noise-related terms. They can be bounded by the energy norm of $ M $ with boundedness but not smallness conditions of the noise. 
	
Small noises are of an independent interest in literature. 
Recent studies \cite{FlandoliGaleatiLuo,Galeati} showed that certain nonlinear PDEs perturbed by spatially divergence free linear transport noises (satisfying suitable $ L^\infty $ and $ L^2 $ conditions), converge weakly to parabolic deterministic equations. In these works, the noise (martingale) part vanishes but Stratonovich correction (in the form $ \Delta \m $) stays and acts as an additional dissipation in the limit under suitable scaling, which can delay blow-ups. It is not known whether a similar result holds for the conventional LLG model (without bi-harmonic interaction) or for our LLG model in higher dimensions. 
	
In addition, to the best of our knowledge, nonlinear transport noise of the form $ \m \times (\xi \cdot \nabla) \m $ (which we dropped) has not been widely explored. For our model, it is inconclusive whether cancellations of highest-order terms can be achieved for Stratonovich and It{\^o} corrections related to this kind of nonlinear noise. Those highest-order terms consist of cross products of mixed derivatives, which seem to require $ H^3 $-estimates in dimension $ d \geq 2 $, i.e., the energy norm is not sufficient. It is therefore another open question regarding the existence of solution when the random part is also in the general non-Galilean form. 
	
In this paper, we show that under certain (spatial) regularity of the noise, there exists a pathwise unique solution of a stochastic LLG equation in the form \eqref{eq: sLLG in intro} for frustrated magnets, with $ H^2 $-moment estimates. 
We formulate the formal stochastic LLG equation \eqref{eq: sLLG with Delta^2} in Section \ref{Section: Notation and the equation} and provide the main results (Theorem \ref{Theorem: E!} and Corollary \ref{Coro: homotopy preserved}) in Section \ref{Section: main results}. 
The rest of the paper is devoted to the proof of Theorem \ref{Theorem: E!}. 
We first construct an approximating equation in Section \ref{Section: Mollification with cut-off} using standard mollifiers and a suitable cut-off function. The latter allows us to estimate certain nonlinear terms in the absence of norm constraint, i.e., when the approximation does not necessarily take values in the unit sphere. 
This becomes clear as we derive $ H^2 $-uniform estimates of the approximation in Section \ref{Section: Uniform estimates of m-eps}, where the norm constraint is narrowly violated due to mollifications. 
Applying Skorohod theorem and compactness embedding results, we deduce strong (resp., weak) convergence in a weighted $ H^1 $ (resp., $ H^2 $) space in Section \ref{Section: convergence for fixed R}. 
Subsequently, we verify that the limit takes values in the unit sphere (Section \ref{Section: vector length of m-tilde}) and prove convergences of the drift and the diffusion terms of the equation separately (Section \ref{Section: convergence of drift and diffusion coefficients}). 
In Section \ref{Section: Proof of Theorem}, we show that the limit is indeed a pathwise unique solution of our stochastic LLG, concluding the proof. 
For the ease of reading, we collect in Appendix \ref{Section: preliminary estimates} some preliminary estimates used for Section \ref{Section: Uniform estimates of m-eps}.
Similar arguments hold for the equation on the torus, for which we give a brief description in Appendix \ref{Section: torus}.

\section{Problem formulation and results}\label{Section: Problem formulation and results}

\subsection{Notation and the equation}\label{Section: Notation and the equation}	
	We denote by $ \S^2 \subset \R^3 $ the unit sphere. 
	Let $ d=2,3 $. 
	Let $ \mathcal{S}(\R^d) $ and $ \mathcal{D}(\R^d) $ denote the Schwartz space of functions and the space of tempered distributions on $ \R^d $, respectively. 
	As usual, we denote by $ \L^p $ and $ {\b W}^{\sigma,p} $ the Lebesgue and Sobolev spaces $ L^p(\R^d; \R^3) $ and $ W^{\sigma,p}(\R^d; \R^3) $, with $ \H^\sigma := {\b W}^{\sigma,2} $. 
	For the weight function $ \rho: \R^d \ni x \mapsto (1+|x|^2)^{-2} \in (0,1] $, 
	\begin{equation}\label{eq: rho^-1 D-D2}
		|\rho^{-1} \partial_i \rho|_{\L^\infty} + |\rho^{-1} \partial_{ij} \rho |_{\L^\infty} \leq c, \quad \forall i,j=1,\ldots,d,
	\end{equation}
	for some constant $ c $.
	Then we define weighted spaces:
	\begin{align*}
		\L^p_\rho &:= \left\{ f: \R^d \to \R^3 \ : \ \left|f\rho^\frac{1}{p} \right|_{\L^p} < \infty \right\}, \quad p \in [1,\infty), \\
		\H^1_\rho &:= \left\{ f \in \L^2_\rho \ : \ \nabla f \in \L^2_\rho \right\}.
	\end{align*}   
	For normed spaces $ E_1 $ and $ E_2 $, we write $ E_1 \hookrightarrow E_2 $ if $ E_1 $ is continuously embedded in $ E_2 $, and $ E_1 \Subset E_2 $ if $ E_1 $ is compactly embedded in $ E_2 $. 
    We will use the notation $\nabla_g$ for the operator $\nabla_g\phi=g\cdot\nabla\phi$.
	
	Let $ h\geq 0 $ 
	and let $ v \in W^{2,\infty}(\R^d;\R^d) $ be a divergence-free spin velocity field, i.e., $ {\rm div} (v) = 0 $.
	Let $ W $ be an $ H^4(\R^d;\R^d) $-valued Wiener process with finite trace-class covariance $ Q $, in the form
	\begin{align*}
		W(t) = \sum_{k=1}^\infty q_k W_k(t)f_k,
	\end{align*}
	where $ \{W_k\} $ is a family of real-valued independent Brownian motions, $ \{ f_k \} $ is a complete orthonormal system in $ H^4(\R^d;\R^d) $ consisting of eigenvectors of $ Q $ such that $ Q f_k = q_k^2 f_k $ for some bounded real $ q_k $ and \smash{$ {\rm Tr} (Q) = \sum_{k=1}^\infty q_k^2 < \infty $}. 
	For every $ k \geq 1 $, let $ g_k := q_k f_k $. We assume that 
    \begin{equation}\label{eq: g in W2inf}
        q^2:=\sum_{k=1}^\infty |g_k|_{H^4(\R^d;\R^d)}^2=\sum_{k=1}^\infty q_k^2<\infty.
    \end{equation}
	Recall \eqref{eq: sLLG in intro}, for simplicity we write $ v $ instead of $ (1+\alpha \beta) v $ and let \smash{$ \gamma := \frac{\alpha - \beta}{1+\alpha \beta} $}. 
	Now we write down the formal stochastic LLG equation. 
	Let $ T \in (0,\infty) $. We will prove the existence of a pathwise unique solution $ \m: [0,T] \times \R^d \to \S^2 $ (see Theorem \ref{Theorem: E!}) to the following equation:
	\begin{equation}\label{eq: sLLG with Delta^2}
		\begin{aligned}
			\d \m(t) 
			&= \m \times (\Delta \m + \Delta^2 \m - he_3) \dt + \alpha \m \times (\m \times (\Delta \m + \Delta^2 \m - he_3)) \dt \\
			&\quad - \la \nabla_v \m + \gamma \m \times \nabla_v m \ra \dt \\
			&\quad + \frac{1}{2} \sum_{k=1}^\infty (\nabla_{g_k})^2 \m \dt - \sum_{k=1}^\infty \nabla_{g_k} \m \dW_k(t), 
		\end{aligned}
	\end{equation}
	where $ \m(0) = \m_0: \R^d \to \S^2 $ and $ \nabla \m_0 \in \H^1 $. 
	For simplicity, 
	we define drift and diffusion coefficients
	\begin{align*}
		\bar{F}(u) &:= u \times (\Delta u + \Delta^2 u - he_3) + \alpha u \times (u \times (\Delta u + \Delta^2 u - he_3)) - \gamma u \times \nabla_v u, \\
		F(u) &:= \bar{F}(u) -\nabla_v u, \\
		S_k(u) &:= (\nabla_{g_k})^2 u, \\
		G_k(u) &:= -\nabla_{g_k} u, 
	\end{align*}
	for any $ u \in \H^2 $, where the bi-Laplacian terms are identified via their weak forms. 
	More explicitly, 
	\begin{equation}\label{eq: cross Delta^2 weak form}
		\begin{aligned}
			\lb w \times \Delta^2 u, \varphi \rb_{\L^2} 
			&= \lb \Delta u, \Delta \varphi \times w + \varphi \times \Delta w + 2 \nabla \varphi \times \nabla w \rb_{\L^2}, \\
			\lb w \times (w \times \Delta^2 u), \varphi \rb_{\L^2}
			&= \lb w \times \Delta u, \varphi \times \Delta w + \Delta \varphi \times w + 2 \nabla \varphi \times \nabla w \rb_{\L^2} 
			+ \lb \Delta u, (\varphi \times w) \times \Delta w \rb_{\L^2} \\
			&\quad + 2 \lb \nabla w \times \Delta u, \nabla \varphi \times w + \varphi \times \nabla w \rb_{\L^2}. \\
		\end{aligned}
	\end{equation}	
	for $ u,w,\varphi \in \H^2 $.

\subsection{Main results}\label{Section: main results}	
	In this section, we state the main theorem for the existence and uniqueness of solution, and prove the preservation of homotopy type as a corollary.  
	\begin{definition}\label{Def: martingale solution}
		Given $ T \in (0,\infty) $, we say that \eqref{eq: sLLG with Delta^2} has a martingale solution if there exists a filtered probability space $ (\Omega, \mathcal{F},(\mathcal{F}_t)_{t \in [0,T]}, \P) $ defined on which an $ H^4(\R^d;\R^d) $-valued Wiener process and a $ (\mathcal{F}_t) $-progressively measurable process $ \m $, such that
		\begin{enumerate}[label=(\roman*), wide, labelwidth=!, labelindent=0pt]
			\ie
			\item 
			$ |\m(t,x)|=1 $, a.e.-$ (t,x) $, $ \P $-a.s. and $ \m \in \mathcal{C}([0,T]; \L^2_\rho) $ satisfies 
			\begin{align*}
				\E \left[ |\nabla \m(t)|_{L^\infty(0,T;\H^1)}^2 + |\m \times \Delta^2 \m|_{L^2(0,T;\L^2)}^2 \right] < \infty,
			\end{align*} 
			
			\item 
			for every $ t \in [0,T] $, the following equality holds $ \P $-a.s. in $ \L^2_\rho $
			\begin{equation}\label{eq: sLLG with Delta^2 int}
				\begin{aligned}
					\m(t)
					&= \m_0 + \int_0^t F(\m(s)) \ds + \sum_{k=1}^\infty \int_0^t G_k(\m(s)) \circ \dW_k(s), 
				\end{aligned}
			\end{equation}	
			where the first Bochner integral and the Stratonovich integral are well-defined in $ \L^2 $.
		\end{enumerate} 
	\end{definition}

	\begin{theorem}\label{Theorem: E!}
		There exists a pathwise unique martingale solution $ (\Omega, \mathcal{F}, (\mathcal{F}_t)_{t \in [0,T]}, \P, W, \m) $ to \eqref{eq: sLLG with Delta^2} in the sense of Definition \ref{Def: martingale solution}, such that for $ p \in [1,\infty) $, 
		\begin{align*}
			\E \left[ |\nabla \m|_{L^\infty(0,T;\H^1)}^{2p} + |\m \times \Delta^2 \m|_{L^2(0,T;\L^2)}^{2p} \right] < \infty, 
		\end{align*}
		and $ \m-e_3 \in \mathcal{C}^\sigma([0,T];\L^2) $ $ \P $-a.s. for $ \sigma \in (0,\frac{1}{2}) $. 
	\end{theorem}

The bounds in Theorem \ref{Theorem: E!} and a simple interpolation argument implies space-time H\"older continuity almost surely.
Hence, $\m(t)$ defines a homotopy of maps from $\mathbb{R}^d$. 
In fact, for $ \frac{d}{2}<\tau <2 $,
\begin{align*}
			\left| \m(s) -\m(t) \right|_{\H^\tau}^2
			\leq c \left|\m(s) -\m(t) \right|_{\L^2}^{2-\tau}
			\left|\m(s) -\m(t) \right|_{\H^2}^{\tau} \lesssim |s-t|^{\sigma(2-\tau)}
\end{align*}
and by Sobolev embedding
\[
|\m(s, x) - \m(t,y)| \le c \left( |\m(s)|_{\H^\tau} |x-y|^{\tau -\frac{d}{2}} + |\m(s)-\m(t)|_{\H^\tau} \right)
\]
with constants $c$ independent of $0 \leq s < t <T$ and $x,y \in \R^d$. 
	
	\begin{corollary}\label{Coro: homotopy preserved}
		Trajectories are continuous in space and time and preserve the homotopy type of initial data almost surely.
	\end{corollary}	
		

\section{Approximating equation}\label{Section: Approximating equation}		
	Since our desired solution $ \m $ is not a function in $ \L^2 $, it is more convenient to define $ m:= \m-e_3 $  and to study the following equation for $m$ in $\L^2$, as in \cite{DoresicMelcher}: 
	\begin{equation}\label{eq: dm}
		\d m(t) 
		= F(m(t)+e_3) \dt + \frac{1}{2}\sum_{k=1}^\infty S_k(m(t)) \dt + \sum_{k=1}^\infty G_k(m(t)) \dW_k(t), 
		\quad
		m(0) = m_0 \in \H^2.
	\end{equation}
	
\subsection{Mollification with cut-off}\label{Section: Mollification with cut-off}
		For \eqref{eq: dm}, we construct an approximating equation using the modified Galerkin method in \cite[Chapter 15, Section 7]{Taylor_book} combined with a cut-off function which controls the vector length of the approximation. 
	
	Let $ J_\epsilon $ denote a Friedrich mollifier on $ \R^d $, such that for $ u \in \mathcal{S}(\R^d) $, 
	\begin{align*}
		J_\epsilon u(x) 
		&:= (j_\epsilon * u)(x) 
		= \int_{\R^d} j_\epsilon(x-y) u(y) \dy
		= \frac{1}{\epsilon^d} \int_{\R^d} j\la \frac{x-y}{\epsilon} \ra u(y) \dy,
	\end{align*}
	where $ j \in \mathcal{S}(\R^d) $ is real-valued and compactly supported with $ \int_{\R^d} j(y) \dy = 1 $. 	
		%
		%
		%
	
	Fix $ R > 1 $. 
	Let $ \psi_R: [0,\infty) \to [0,1] $ be a smooth non-increasing cut-off function such that $ \psi_R(y) = 1 $ for $ y \in [0,R] $ and $ \psi_R(y)=0 $ for $ y \geq R+1 $.  
	Then for a sufficiently smooth function $ u: [0,T] \times \R^d \to \R^3 $, 
	\begin{align*}
		\frac{1}{2} \psi_R(|u(t,x)|^2) |u(t,x)|^2 \leq R, \quad \forall (t,x) \in [0,T] \times \R^d. 
	\end{align*} 
	Let $ f^{(i)} $ denote the $ i $-th derivative of a smooth function $ f: \R \to \R $. 
	We have
	\begin{equation}\label{eq: grad psi}
		\begin{aligned}
			\frac{1}{2}\nabla \psi_R(|u|^2)
			&= \psi_R^{(1)}(|u|^2) \lb u, \nabla u \rb, \\
			\frac{1}{2} \Delta \psi_R(|u|^2)
			&= \psi_R^{(1)}(|u|^2) \la |\nabla u|^2 + \lb u, \Delta u \rb \ra 
			+ 2\psi_R^{(2)}(|u|^2) \lb u, \nabla u \rb^2.
		\end{aligned}
	\end{equation}
	Moreover, for $ p \geq 0 $ and any non-negative integer $ i $, there exists $ c(R)= c(R,p) $ such that
	\begin{align*}
		\psi_R^{(i)}(|u(t,x)|^2) |u(t,x)|^p \leq c(R), \quad \forall (t,x) \in [0,T] \times \R^d,
	\end{align*}
	which implies that on $ [0,T] \times \R^d $, 
	\begin{equation}\label{eq: grad psi |u|^p}
		\begin{aligned}
			\nabla \psi_R(|u|^2) |u|^p &\leq c(R) |\nabla u|, \\
			\Delta \psi_R(|u|^2) |u|^p &\leq c(R) \la |\nabla u|^2 + |\Delta u| \ra. 
		\end{aligned}
	\end{equation}
	
	For $ \epsilon > 0 $, consider the approximating equation:
	\begin{equation}\label{eq: dm-epsilon-R}
		\d m_\epsilon(t)
		= F^R_\epsilon(m_\epsilon(t)) \dt + \frac{1}{2} \sum_{k=1}^\infty S_{k,\epsilon}(m_\epsilon(t)) \dt + \sum_{k=1}^\infty G_{k,\epsilon}(m_\epsilon(t)) \dW_k(t), 
		\quad
		m_\epsilon(0) = J_\epsilon m_0, 
	\end{equation}
	where 
	\begin{align*}
		F^R_\epsilon(m_\epsilon) &= J_\epsilon \la \psi_R(|J_\epsilon m_\epsilon+e_3|^2) \bar{F}(J_\epsilon m_\epsilon + e_3) -  \nabla_v J_\epsilon m_\epsilon \ra, \\
		S_{k,\epsilon}(m_\epsilon) &= J_\epsilon S_k(J_\epsilon m_\epsilon), \\
		G_{k,\epsilon} (m_\epsilon) &= J_\epsilon G_k(J_\epsilon m_\epsilon).
	\end{align*}
	The cut-off function $ \psi_R $ is only required to address the nonlinear term $ \bar{F} $. 
	For the linear terms, we integrate-by-parts to obtain desired estimates as shown in the proof of Lemma \ref{Lemma: meps H2-est}.

\subsection{Uniform estimates of $ m_\epsilon $}\label{Section: Uniform estimates of m-eps}
	We first show that the approximating equation \eqref{eq: dm-epsilon-R} admits a unique solution $ m_\epsilon $ in Lemma \ref{Lemma: dm-epsilon-R has soln}, and then deduce uniform estimates of $ m_\epsilon $ in Lemma \ref{Lemma: meps H2-est}. At the end of this section, we estimate the deviation of $ m_\epsilon+e_3 $ from the unit sphere in Lemma \ref{Lemma: |meps+e3| <= remainder}. 
	
	\begin{lemma}\label{Lemma: dm-epsilon-R has soln}
		For every $ \epsilon>0 $, there exists a unique solution $ m_\epsilon $ of \eqref{eq: dm-epsilon-R}, where $ m_\epsilon $ is a progressively measurable process taking values in $ \H^2 $, such that 
		\begin{align*}
			\sup_{t \in [0,T]} \E \left[ |m_\epsilon(t)|_{\H^2}^2 \right] < \infty, 
		\end{align*}
		and for every $ t \in [0,T] $, the following equality holds $ \P $-a.s. in $ \L^2 $
		\begin{align*}
			m_\epsilon(t)
			= J_\epsilon m_0 + \int_0^t F^R_\epsilon(m_\epsilon(s)) \ds + \frac{1}{2} \sum_{k=1}^\infty \int_0^t S_{k,\epsilon}(m_\epsilon(s)) \ds + \sum_{k=1}^\infty \int_0^t G_{k,\epsilon}(m_\epsilon(s)) \dW_k(s). 
		\end{align*}
	\end{lemma}		
	\begin{proof}
		Fix $ \epsilon > 0 $. 
		We first verify that $ F_\epsilon^R $, $ S_{k,\epsilon} $ and $ G_{k,\epsilon} $ are locally Lipschitz on $ \H^2 $. 
		Let $ u,w \in \H^2 $. 
		The derivatives of $ j $ are in $ \mathcal{S}(\R^2) $, bounded, and in $ \L^1 $. 
		Thus, for any non-negative integer $ \sigma $,
		\begin{align*}
			|\nabla^\sigma J_\epsilon u|_{\H^2} 
			&= |(\nabla^\sigma j_\epsilon)*(u + \nabla u + \Delta u)|_{\L^2} \\
			&\leq |\nabla^\sigma j_\epsilon|_{\L^1} |u + \nabla u + \Delta u|_{\L^2} \\
			&\leq c(\epsilon) |u|_{\H^2},
		\end{align*}
		and similarly, 
		\begin{align*}
			|\nabla^\sigma J_\epsilon(u-w)|_{\H^2} 
			&\leq c(\epsilon) |u-w|_{\H^2}.
		\end{align*}
		Recall that $ \H^2 \hookrightarrow \L^\infty $. Let $ f $ be a locally Lipschitz function on $ \H^2 $ with $ f(0)=0 $, then
		\begin{align*}
			|u \times f(u) - w \times f(w) |_{\H^2} 
			&= |(u-w) \times f(u)|_{\H^2} + |w \times (f(u)-f(w)) |_{\H^2} \\
			&\lesssim |u-w|_{\H^2} |f(u)|_{\H^2} + |w|_{\H^2} |f(u)-f(w)|_{\H^2}.
		\end{align*}
		Similar arguments follow for scalar products. 
		Then with the Lipschitz property of $ J_\epsilon $ and $ \psi_R $, it is clear that $ F_\epsilon^R $, $ G_{k,\epsilon} $ and $ S_{k,\epsilon} $ are locally Lipschitz on $ \H^2 $ for $ k \geq 1 $.
		
		Next, let $ \mathcal{E} $ denote the space of $ \H^2 $-valued progressively measurable processes, with the norm 
		\begin{align*}
			|u|_{\mathcal{E}}^2 = \sup_{t \in [0,T]} \E \left[ |u(t)|_{\H^2}^2 \right]. 
		\end{align*}
		For $ n \in {\b N} $, let $ F^R_{\epsilon,n} $, $ S_{k,\epsilon,n} $ and $ G_{k,\epsilon,n} $ denote the Lipschitz modifications of $ F_\epsilon^R $, $ S_{k,\epsilon} $ and $ G_{k,\epsilon} $ on $ \H^2 $, respectively. For example, we set 
		\begin{align*}
			F^R_{\epsilon,n}(u) 
			= \begin{cases} 
				F_\epsilon^R(u) &\text{ if } |u|_{\H^2} \leq n, \\
				F_\epsilon^R\la \frac{n u}{|u|_{\H^2}} \ra &\text{ if } |u|_{\H^2} > n. 
			\end{cases}
		\end{align*}
		Let $ A_n : \mathcal{E} \to \mathcal{E} $ be given by
		\begin{align*}
			A_n(u)(t) 
			&= m_0 + \int_0^t \la F^R_{\epsilon,n}(u(s)) + \frac{1}{2} \sum_{k=1}^\infty S_{k,\epsilon,n}(u(s)) \ra \ds + \sum_{k=1}^\infty \int_0^t G_{k,\epsilon,n}(u(s)) \dW_k(s) \\
			&= m_0 + I_n(t) + M_n(t), 
		\end{align*}
		where $ I_n, M_n \in \mathcal{E} $ by the Lipschitz continuity of $ F^R_{\epsilon,n} $, $ S_{k,\epsilon,n} $ and $ G_{k,\epsilon,n} $. 
		In particular, $ M_n $ is an $ \H^2 $-valued continuous martingale. 
		Once we verify the Lipschitz property of $ A_n $ on $ \mathcal{E} $, standard arguments using Banach fixed point theorem and localisation by stopping times show that there exists a unique solution $ m_{\epsilon} $ in $ \mathcal{E} $ (the limit over $ n $) to the $ \epsilon $-approximating equation \eqref{eq: dm-epsilon-R}.  
	\end{proof}

	In the following lemma, we deduce an $ \H^2 $-estimate of $ m_\epsilon $ which dominates the energy norm $ E(m_\epsilon+e_3) $ but allows us to directly apply Gronwall's inequality to obtain desired estimates. 
	\begin{lemma}\label{Lemma: meps H2-est}
		Let $ |m_0|_{\H^2} \leq c_0 $. 
		Then for $ p \in [1,\infty) $, there exists a constant $ c = c(\tr{c_0,}p,T) $ independent of $ \epsilon $ and $ R $ such that
		\begin{align*}
			\E \left[ \sup_{t \in [0,T]} |m_\epsilon(t)|_{\H^2}^{2p} \right] \leq c e^{cR^{p}},
		\end{align*}
		and for $ \psi^R_\epsilon := \psi_R(|J_\epsilon m_\epsilon+e_3|^2) $, 
		\begin{align*}
			\E \left[ \la \int_0^T |(\psi_\epsilon^R)^\frac{1}{2} (J_\epsilon m_\epsilon+e_3) \times \Delta^2 J_\epsilon m_\epsilon|_{\L^2}^2(t) \dt \ra^p \right] \leq c(1+R^{p}) e^{cR^{p}}.
		\end{align*}
	\end{lemma}
	\begin{proof}
		Let $ u_\epsilon:= J_\epsilon m_\epsilon $ and note that $ \langle u, J_\epsilon w \rangle_{\L^2} = \langle J_\epsilon u, w \rangle_{\L^2} $ for any $ u,w \in \L^2 $. 
		Let $ \psi^{R,(i)}_\epsilon := \psi_R^{(i)} (|J_\epsilon m_\epsilon+e_3|^2) $. 
		We divide the proof into several steps.

		\textbf{Step 1.} Estimate $ \frac{1}{2}|m_\epsilon(t)|^2_{\L^2} $.
		\begin{align*}
			\frac{1}{2}\d|m_\epsilon(t)|^2_{\L^2} 
			&= 
			\lb m_\epsilon, F^R_\epsilon(m_\epsilon) \rb_{\L^2} \dt 
			+ \frac{1}{2} \sum_k \la \lb m_\epsilon, S_{k,\epsilon}(m_\epsilon) \rb_{\L^2} + |G_{k,\epsilon}(m_\epsilon)|^2_{\L^2} \ra \dt \\
			&\quad + \sum_k \lb m_\epsilon, G_{k,\epsilon}(m_\epsilon) \rb_{\L^2} \dW_k(t) \\
			&=: {\rm I}_1 \dt + \frac{1}{2} \sum_k {\rm I}_{2,k} \dt + \sum_k {\rm I}_{3,k} \dW_k(t). 
		\end{align*}
		The assumption $ {\rm div} v = 0 $ and \eqref{eq: <u, Dfu> simplify} yield $\langle u_\epsilon, \nabla_v u_\epsilon \rangle_{\L^2} = 0 $. 
		Then for $ \delta \in (0,1) $ and $ R>1 $, we have
		\begin{align*}
			{\rm I}_1
			&= \lb m_\epsilon, F^R_\epsilon(m_\epsilon) \rb_{\L^2} \\
			&= \lb u_\epsilon, \psi^R_\epsilon \bar{F}(u_\epsilon +e_3) - \nabla_v u_\epsilon \rb_{\L^2} \\ 
			&= \lb u_\epsilon, \psi^R_\epsilon (u_\epsilon+e_3) \times \la \Delta u_\epsilon + \alpha (u_\epsilon + e_3) \times (\Delta u_\epsilon - h e_3) \ra \rb_{\L^2} \\ 
			&\quad + \lb u_\epsilon, \psi^R_\epsilon (u_\epsilon + e_3) \times \la \Delta^2 u_\epsilon + \alpha (u_\epsilon + e_3) \times \Delta^2 u_\epsilon \ra \rb_{\L^2} \\
			&\quad - \gamma \lb u_\epsilon, \psi^R_\epsilon (u_\epsilon + e_3) \times \nabla_v u \rb_{\L^2} \\
			&\quad -\lb u_\epsilon, \nabla_v u_\epsilon \rb_{\L^2} \\
			&\leq
			\delta |(\psi^R_\epsilon)^\frac{1}{2}(u_\epsilon + e_3) \times \Delta^2 u_\epsilon|_{\L^2}^2 + c(\delta) R|u_\epsilon|_{\H^2}^2.	
		\end{align*}
		
		The Stratonovich part has a similar term with opposite sign to the It{\^o} correction. Integrating-by-parts, 
		\begin{align*}
			\lb m_\epsilon, S_{k,\epsilon}(m_\epsilon) \rb_{\L^2}
			&= \lb u_\epsilon, (\nabla_{g_k})^2 u_\epsilon \rb_{\L^2} \\
			&= -\lb \nabla_{g_k} u_\epsilon + ({\rm div} g_k) u_\epsilon, \nabla_{g_k} u_\epsilon \rb_{\L^2} \\
			&\leq -|\nabla_{g_k} u_\epsilon|_{\L^2}^2 + |g_k|_{W^{1,\infty}(\R^d;\R^d)}^2 |u_\epsilon|_{\L^2} |\nabla u_\epsilon|_{\L^2},
		\end{align*}
		where $ |G_{k,\epsilon}(m_\epsilon)|^2_{\L^2} = |J_\epsilon \nabla_{g_k} u_\epsilon|_{\L^2}^2 \leq |\nabla_{g_k} u_\epsilon|_{\L^2}^2 $, leaving
		\begin{align*}
			\sum_k {\rm I}_{2,k} 
			&\lesssim \sum_k |g_k|_{W^{1,\infty}(\R^d;\R^d)}^2 |u_\epsilon|_{\H^1}^2. 
		\end{align*} 
		Similarly, for the diffusion term, by \eqref{eq: <u, Dfu> simplify}, 
		\begin{align*}
			{\rm I}_{3,k} 
			= -\lb u_\epsilon, \nabla_{g_k} u_\epsilon \rb_{\L^2} 
			&= \frac{1}{2} \int_{\R^d} ({\rm div} g_k) |u_\epsilon|^2 \dx \\
			&\leq \frac{1}{2} |g_k|_{W^{1,\infty}(\R^d;\R^d)} |u_\epsilon|_{\L^2}^2,
		\end{align*}
		and by the Burkholder-Davis-Gundy inequality, for $ p \in [1,\infty) $, 
		\begin{align*}
			\E \left[ \sup_{s \in [0,t]} \left| 2\int_0^s \sum_k {\rm I}_{3,k} \dW_k(r) \right|^p \right]
			&\leq cb_p \E \left[ \la \sum_k |g_k|_{W^{1,\infty}(\R^d;\R^d)}^2 \int_0^t |u_\epsilon|_{\L^2}^4 \ds \ra^\frac{p}{2} \right] \\
			&\leq cb_p \E \left[ \sup_{s \in [0,t]} |u_\epsilon(s)|_{\L^2}^p \la \sum_k |g_k|_{W^{1,\infty}(\R^d;\R^d)}^2 \int_0^t |u_\epsilon|_{\L^2}^2 \ds \ra^\frac{p}{2} \right] \\
			&\leq \delta^p \E \left[ \sup_{s \in [0,t]} |u_\epsilon(s)|_{\L^2}^{2p} \right] \\
			&\quad + c\delta^{-p} b_p^2 \la \sum_k |g_k|_{W^{1,\infty}(\R^d;\R^d)}^2 \ra^p \E \left[\int_0^t \sup_{r \in [0,s]} |u_\epsilon(r)|_{\L^2}^{2p} \ds \right].
		\end{align*}

		\textbf{Step 2.} Estimate $ \frac{1}{2}|\nabla m_\epsilon(t)|^2_{\H^1} $.
		\begin{align*}
			\frac{1}{2} \d \la |\nabla m_\epsilon|_{\L^2}^2 + |\Delta m_\epsilon|_{\L^2}^2 \ra 
			&= \lb -\Delta m_\epsilon + \Delta^2 m_\epsilon, F^R_\epsilon(m_\epsilon) \rb_{\L^2} \dt \\
			&\quad + \frac{1}{2} \sum_k \la \lb -\Delta m_\epsilon, S_{k,\epsilon}(m_\epsilon) \rb_{\L^2} + |\nabla G_{k,\epsilon}(m_\epsilon)|^2_{\L^2} \ra \dt \\
			&\quad + \frac{1}{2} \sum_k \la \lb \Delta^2 m_\epsilon, S_{k,\epsilon} (m_\epsilon) \rb_{\L^2} + |\Delta G_{k,\epsilon} (m_\epsilon)|^2_{\L^2} \ra \dt \\
			&\quad + \sum_k \lb -\Delta m_\epsilon + \Delta^2 m_\epsilon, G_{k,\epsilon}(m_\epsilon) \rb_{\L^2} \dW_k(t) \\
			&=: {\rm II}_{1} \dt + \frac{1}{2} \sum_k {\rm II}_{2,k} \dt + \sum_k {\rm II}_{3,k} \dW_k(t). 
		\end{align*}
		We address each term in the following calculations (i) -- (iii). 	
		\begin{enumerate}[label=(\roman*), wide, labelwidth=!, labelindent=0pt]
			\item 
			Estimate $ {\rm II}_{1} $. 
			\begin{align*}
				{\rm II}_{1}
				&= \lb -\Delta m_\epsilon + \Delta^2 m_\epsilon, F^R_\epsilon(m_\epsilon) \rb_{\L^2} \\
				&= \lb -\Delta u_\epsilon + \Delta^2 u_\epsilon, \psi^R_\epsilon \bar{F}(u_\epsilon +e_3) - \nabla_v u_\epsilon \rb_{\L^2} \\ 
				&= \lb -\Delta u_\epsilon, \psi^R_\epsilon (u_\epsilon+e_3) \times \Delta^2 u_\epsilon \rb_{\L^2} 
				+ \lb \Delta^2 u_\epsilon, \psi^R_\epsilon (u_\epsilon+e_3) \times \Delta u_\epsilon \rb_{\L^2} \\
				&\quad + \alpha |(\psi^R_\epsilon)^\frac{1}{2} (u_\epsilon+e_3) \times \Delta u_\epsilon|_{\L^2}^2 
				+ \alpha \lb -\Delta u_\epsilon, \psi^R_\epsilon (u_\epsilon +e_3) \times \la (u_\epsilon+e_3) \times \Delta^2 u_\epsilon \ra \rb_{\L^2} \\
				&\quad - \alpha |(\psi^R_\epsilon)^\frac{1}{2} (u_\epsilon+e_3) \times \Delta^2 u_\epsilon|_{\L^2}^2 
				+ \alpha \lb \Delta^2 u_\epsilon, \psi^R_\epsilon (u_\epsilon +e_3) \times \la (u_\epsilon+e_3) \times \Delta u_\epsilon \ra \rb_{\L^2} \\
				&\quad + \lb \Delta u_\epsilon - \Delta^2 u_\epsilon, h \psi^R_\epsilon (u_\epsilon + e_3) \times \la -u_\epsilon + (u_\epsilon \times e_3) \ra \rb_{\L^2} \\
				&\quad + \gamma \lb \Delta u_\epsilon - \Delta^2 u_\epsilon, \psi^R_\epsilon (u_\epsilon+e_3) \times \nabla_v u_\epsilon \rb_{\L^2} \\
				&\quad + \lb \Delta u_\epsilon - \Delta^2 u_\epsilon, \nabla_v u_\epsilon \rb_{\L^2} \\
				&\leq
				-(\alpha-\delta) |(\psi^R_\epsilon)^\frac{1}{2} (u_\epsilon+e_3) \times \Delta^2 u_\epsilon|_{\L^2}^2 
				+ c(\delta) \la |u_\epsilon|_{\L^2}^2 + |v|_{L^\infty(\R^d;\R^d)} |\nabla u_\epsilon|_{\L^2}^2 + (1+R)|\Delta u_\epsilon|_{\L^2}^2 \ra \\
				&\quad + \lb \Delta u_\epsilon - \Delta^2 u_\epsilon, \nabla_v u_\epsilon \rb_{\L^2},
			\end{align*}
			where
			\begin{align*}
				\lb \Delta u_\epsilon, \nabla_v u_\epsilon \rb_{\L^2} 
				&\leq |v|_{L^\infty(\R^d;\R^d)} |\nabla u_\epsilon|_{\L^2} |\Delta u_\epsilon|_{\L^2}, 
			\end{align*}
			and with $ {\rm div} v = 0 $, we have $ \langle \Delta u_\epsilon, \nabla_v \Delta u_\epsilon \rangle_{\L^2} = 0 $ and 
			\begin{align*}
				\lb \Delta^2 u_\epsilon, \nabla_v u_\epsilon \rb_{\L^2} 
				&= \lb \Delta u_\epsilon, \Delta \nabla_v u_\epsilon \rb_{\L^2} \\
				&= \lb \Delta u_\epsilon, \nabla_v \Delta u_\epsilon + \nabla_{\Delta v} u_\epsilon \rb_{\L^2} + 2 \sum_{i,j=1}^d \lb \Delta u_\epsilon, \partial_i v_j \partial_{ij} u_\epsilon \rb_{\L^2} \\
				&\lesssim 
				|v|_{W^{2,\infty}(\R^d;\R^d)} |\nabla u_\epsilon|_{\H^1}^2.	
			\end{align*}
			Thus, 
			\begin{align*}
				{\rm II}_1 &\leq 
				-(\alpha-\delta) |(\psi^R_\epsilon)^\frac{1}{2} (u_\epsilon+e_3) \times \Delta^2 u_\epsilon|_{\L^2}^2 \\
				&\quad + c(\delta)(1+ |v|_{W^{2,\infty}(\R^d;\R^d)} ) \la |u_\epsilon|_{\H^1}^2 + (1+R)|\Delta u_\epsilon|_{\L^2}^2 \ra. 
			\end{align*}

			\item 
			Estimate $ {\rm II}_{2,k} $. 
			\begin{align*}
				{\rm II}_{2,k}
				&= \la -\lb \Delta m_\epsilon, S_{k,\epsilon}(m_\epsilon) \rb_{\L^2} 
				+ |\nabla G_{k,\epsilon}(m_\epsilon)|^2_{\L^2} \ra 
				+ \la \lb \Delta^2 m_\epsilon, S_{k,\epsilon}(m_\epsilon) \rb_{\L^2} 
				+ |\Delta G_{k,\epsilon}(m_\epsilon)|^2_{\L^2} \ra \\
				&=: {\rm II}_{2a,k} + {\rm II}_{2b,k}. 
			\end{align*}
			For the It{\^o} corrections, note that $ J_\epsilon $ commutes with $ \nabla $ and $ \Delta $. 
			We have
			\begin{align*}
				{\rm II}_{2a,k}
				&= -\lb \Delta u_\epsilon, (\nabla_{g_k})^2 u_\epsilon \rb_{\L^2} 
				+ |J_\epsilon \nabla (\nabla_{g_k} u_\epsilon) |^2_{\L^2} \\
				&\lesssim |g_k|^2_{W^{1,\infty}(\R^d;\R^d)} |\nabla u_\epsilon|_{\H^1}^2. 
			\end{align*}
			For $ {\rm II}_{2b,k} $, 
			we first re-write the Stratonovich term: 
			\begin{align*}
				\lb \Delta^2 m_\epsilon, S_{k,\epsilon}(m_\epsilon) \rb_{\L^2} 
				&= \lb \Delta^2 u_\epsilon, (\nabla_{g_k})^2 u_\epsilon \rb_{\L^2}. 
			\end{align*}
			Then we estimate the $ \Delta G_{k,\epsilon} $ part. 
			By Lemma \ref{Lemma: D4 vs D1 estimates} and \eqref{eq: T1a, T1b, T1c, <D4u,D1fu>}, 
			taking $ u = u_\epsilon $ and $ f = g_k $, we obtain 
			\begin{align*}
				|\Delta G_{k,\epsilon}(m_\epsilon)|^2_{\L^2}
				&= |J_\epsilon \Delta \nabla_{g_k} u_\epsilon|_{\L^2}^2 \\
				&\leq |\Delta \nabla_{g_k} u_\epsilon|_{\L^2}^2 \\
				&= \lb \Delta^2 \nabla_{g_k} u_\epsilon, \nabla_{g_k} u_\epsilon \rb_{\L^2} \\
				&= \lb \nabla_{g_k} \Delta^2 u_\epsilon + 4{\rm T}_{1a}(u_\epsilon) + 2{\rm T}_{1b}(u_\epsilon) + {\rm T}_{1c}(u_\epsilon), \nabla_{g_k} u_\epsilon \rb_{\L^2} \\
				&= 
				-\lb \Delta^2 u_\epsilon, (\nabla_{g_k})^2 u_\epsilon \rb_{\L^2} 
				-\lb \Delta^2 u_\epsilon, ({\rm div} g_k) \nabla_{g_k} u_\epsilon \rb_{\L^2} \\
				&\quad + \lb 4{\rm T}_{1a}(u_\epsilon) + 2{\rm T}_{1b}(u_\epsilon) + {\rm T}_{1c}(u_\epsilon), \nabla_{g_k} u_\epsilon \rb_{\L^2} \\
				&\leq 
				-\lb \Delta^2 m_\epsilon, S_{k,\epsilon}(m_\epsilon) \rb_{\L^2} 
				+ c |g_k|^2_{H^4(\R^d;\R^d)} |\nabla u_\epsilon|^2_{\H^1}.
			\end{align*}	
			Thus, 
			\begin{align*}
				\sum_k {\rm II}_{2,k} 
				&\lesssim \sum_k |g_k|^2_{H^4(\R^d;\R^d)} |\nabla u_\epsilon|^2_{\H^1}
			\end{align*}

			\item 
			Estimate the diffusion part $ {\rm II}_{3,k} $. 
			\begin{align*}
				{\rm II}_{3,k} 
				&= -\lb \Delta m_\epsilon, G_{k,\epsilon}(m_\epsilon) \rb_{\L^2} 
				+ \lb \Delta^2 m_\epsilon, G_{k,\epsilon}(m_\epsilon) \rb_{\L^2} \\
				&=: {\rm II}_{3a,k} + {\rm II}_{3b,k} .
			\end{align*}
			We have
			\begin{align*}
				{\rm II}_{3a,k}
				&= \lb \Delta u_\epsilon, \nabla_{g_k} u_\epsilon \rb_{\L^2}
				\leq |g_k|_{L^\infty(\R^d;\R^d)} |\Delta u_\epsilon|_{\L^2} |\nabla u_\epsilon|_{\L^2}, 
			\end{align*}
			and 
			\begin{align*}
				{\rm II}_{3b,k}
				&= -\lb \Delta^2 u_\epsilon, \nabla_{g_k} u_\epsilon \rb_{\L^2} \\
				&= -\lb \Delta u_\epsilon, \Delta \nabla_{g_k} u_\epsilon \rb_{\L^2} \\
				&= -\lb \Delta u_\epsilon, \nabla_{g_k} \Delta u_\epsilon + \nabla_{\Delta g_k} u_\epsilon \rb_{\L^2} 
				- 2\sum_{i,j=1}^d \lb \Delta u_\epsilon, \partial_i g_j \partial_{ij} u_\epsilon \rb_{\L^2} \\
				&\leq 
				-\lb \Delta u_\epsilon, \nabla_{g_k} \Delta u_\epsilon \rb_{\L^2} + c |g_k|_{W^{2,\infty}(\R^d;\R^d)} \la |\Delta u_\epsilon|_{\L^2} |\nabla u_\epsilon|_{\L^2} + |\Delta u_\epsilon|_{\L^2}^2 \ra, 
			\end{align*}
			where
			\begin{align*}
				-\lb \Delta u_\epsilon, \nabla_{g_k} \Delta u_\epsilon \rb_{\L^2}
				&= \lb \Delta u_\epsilon, \nabla_{g_k} \Delta u_\epsilon + ({\rm div} g_k) \Delta u_\epsilon \rb_{\L^2} \\
				&= \frac{1}{2}\lb \Delta u_\epsilon, ({\rm div} g_k) \Delta u_\epsilon \rb_{\L^2} \\
				&\lesssim |g_k|_{W^{1,\infty}(\R^d;\R^d)} |\Delta u_\epsilon|_{\L^2}^2.
			\end{align*}
			Then by Burkholder-Davis-Gundy inequality, for $ p \in [1,\infty) $, 
			\begin{align*}
				\E \left[ \sup_{s \in [0,t]} \left| 2 \int_0^s \sum_k {\rm II}_{3,k} \dW_k(r) \right|^p \right] 
				&\leq cb_p \E \left[ \la \int_0^t \sum_k |g_k|^2_{W^{2,\infty}(\R^d;\R^d)} |\Delta u_\epsilon|^2_{\L^2} |\nabla u_\epsilon|^2_{\H^1} \ds \ra^\frac{p}{2} \right] \\
				&\leq \delta^p \E \left[ \sup_{s \in [0,t]} |\nabla u_\epsilon(s)|^{2p}_{\H^1} \right] \\
				&\quad + c \delta^{-p} b_p^2 \la \sum_k |g_k|^2_{W^{2,\infty}(\R^d;\R^d)} \ra^p \E \left[ \int_0^t \sup_{r \in [0,s]}|\Delta u_\epsilon(r)|_{\L^2}^{2p} \ds \right].
			\end{align*}
			
		\end{enumerate}

		\textbf{Step 3.} Combine estimates of $ \frac{1}{2} |m_\epsilon(t)|_{\H^2}^2 $. 
		
		Recall the assumption \eqref{eq: g in W2inf}. 
		For $ R>1 $, we have 
		\begin{align*}
			&|m_\epsilon(t)|_{\H^2}^2 + 2(\alpha-2\delta)|(\psi_\epsilon^R)^\frac{1}{2} (u_\epsilon+e_3) \times \Delta^2 u_\epsilon|^2_{\L^2}(t) \\
			&\leq |J_\epsilon m_0|_{\H^2}^2
			+ c \int_0^t R|u_\epsilon(s)|_{\H^2}^2 \ds 
			+ 2\sup_{s \in [0,t]} \left| \int_0^s \sum_k ({\rm I}_{3,k} + {\rm II}_{3,k}) \dW_k(r) \right|,
		\end{align*}
		where $ c $ depends on $ |v|_{W^{2,\infty}(\R^d;\R^d)}, \sum_k |g_k|^2_{H^4(\R^d;\R^d)}, \delta^{-1} , T $, but not on $ n $ or $ R $. 
		
		Let $ \delta \in (0,\frac{1}{4}) $ be sufficiently small such that $ \alpha - 2\delta > \frac{1}{2} $. 
		Then since $ |u_\epsilon|_{\H^\sigma} \leq |m_\epsilon|_{\H^\sigma} $ for any $ \sigma \geq 0 $, we have for $ p \in [1,\infty) $ that
		\begin{align*}
			&\E \left[ \sup_{t \in [0,T]} |m_\epsilon(t)|_{\H^2}^{2p} + \la \int_0^T |(\psi_\epsilon^R)^\frac{1}{2} (J_\epsilon m_\epsilon+e_3) \times \Delta^2 J_\epsilon m_\epsilon|^2_{\L^2}(t) \dt \ra^p \right] \\
			&\leq 
			c\E \left[ |J_\epsilon m_0|_{\H^2}^{2p} 
			+ \int_0^T R^p \sup_{s \in [0,t]} |m_\epsilon(s)|_{\H^2}^{2p} \dt \right] 
			+ 2^{p} \delta^p \E \left[ \sup_{t \in [0,T]} |m_\epsilon(t)|^{2p}_{\H^2} \right],
		\end{align*}
		where $ 1-2^p\delta^p > \frac{1}{2} $, and then the last expectation term on the right-hand side can be absorbed into the left-hand side.  
		By Gronwall's inequality, 
		\begin{align*}
			\E \left[ \sup_{t \in [0,T]} |m_\epsilon(t)|_{\H^2}^{2p} \right]
			&\leq c \E \left[ |m_0|_{\H^2}^{2p}  \right] e^{c R^{p}}
			\leq c e^{cR^{p}},
		\end{align*}
		for some constant $ c $ independent of $ n $ and $ R $. 
		As a result, 
		\begin{align*}
			\E \left[ \la \int_0^T |(\psi_\epsilon^R)^\frac{1}{2} (J_\epsilon m_\epsilon+e_3) \times \Delta^2 J_\epsilon m_\epsilon|^2_{\L^2}(t) \dt \ra^p \right]
			&\leq c(1+R^{p}) e^{cR^{p}},
		\end{align*}
		concluding the proof.
	\end{proof}

	As a result of Lemma \ref{Lemma: meps H2-est}, 
	$ F^R_\epsilon(m_\epsilon) + \frac{1}{2}\sum_k S_{k,\epsilon}(m_\epsilon) \in L^{p}(\Omega;L^2(0,T;\L^2)) $, and 
	\begin{align*}
		\E \left[ \left| \int_0^T \sum_k G_{k,\epsilon}(m_\epsilon(t)) \dW_k(t) \right|_{W^{\sigma,p}(0,T;\L^2)}^p \right]
		&\lesssim \E \left[ \la \int_0^T \sum_k |g_k|_{L^\infty(\R^d;\R^d)}^2 |\nabla m_\epsilon(t)|_{\L^2}^2 \dt \ra^\frac{p}{2} \right] 
		\leq c(R),
	\end{align*}
	for fixed $ R>1 $, $ \sigma \in (0,\frac{1}{2}) $, $ p \in [2,\infty) $. 
	Thus, with the embedding $ W^{1,2}(0,T;\L^2) \hookrightarrow W^{\sigma,p}(0,T;\L^2) $ for $ \sigma - \frac{1}{p} < \frac{1}{2} $, we have
	\begin{equation}\label{eq: meps in Wsigma,p}
		\E \left[ |m_\epsilon|_{W^{\sigma,p}(0,T;\L^2)}^p \right] \leq c(R). 
	\end{equation}

	Next we examine the vector length of $ M_\epsilon:= m_\epsilon+e_3 $. 
	In Lemma \ref{Lemma: |meps+e3| <= remainder}, we show that although $ M_\epsilon $ does not take values in $ \S^2 $, its deviation from $ \S^2 $ can be controlled by the strength of mollifications.  
	The proof is similar to that in \cite{Melcher_LLG_highdim} which starts from applying It{\^o}'s formula to 
	\begin{equation}\label{eq: defn varphi(meps)}
		\varphi(m_\epsilon) := \frac{1}{4}|1-|m_\epsilon+e_3|^2|_{\L^2_\rho}^2 = \frac{1}{4}|1-|M_\epsilon|^2|_{\L^2_\rho}^2.
	\end{equation}
	Here we work with mollified functions and weighted spaces. 
	Thus, we first identify some remainder terms which arise from the difference between the convolution of products and the product of convolutions, before stating and proving Lemma \ref{Lemma: |meps+e3| <= remainder}. 
	Again, let $ u_\epsilon := J_\epsilon m_\epsilon $, and $ d_\epsilon(f) := J_\epsilon f - f $ for any $ f: [0,T]\times \R^d \to \R $ or $ \R^3 $. 
	Then 
	\begin{equation}\label{eq: defn peps, qeps}
		\begin{aligned}
			p_\epsilon(m_\epsilon)
			&:= J_\epsilon \la (1-|m_\epsilon+e_3|^2) (m_\epsilon+e_3) \rho \ra - (1-|J_\epsilon m_\epsilon+e_3|^2) (J_\epsilon m_\epsilon+e_3) \rho \\
			&= J_\epsilon \la (1-|M_\epsilon|^2) M_\epsilon \rho \ra - (1-|u_\epsilon+e_3|^2) (u_\epsilon+e_3) \rho \\
			&= d_\epsilon((1-|M_\epsilon|^2) M_\epsilon \rho) 
			- (1-|u_\epsilon+e_3|^2) d_\epsilon(m_\epsilon) \rho 
			+ \lb d_\epsilon(m_\epsilon), m_\epsilon+ u_\epsilon + 2e_3 \rb M_\epsilon \rho.
		\end{aligned}  
	\end{equation}
	Since $ |d_\epsilon(f(t))|_{\L^2} \lesssim |f(t)|_{\L^2} $ for any $ t \in [0,T] $, using Lemma \ref{Lemma: meps H2-est} and the continuous embeddings $ \H^2 \hookrightarrow \L^\infty $ and $ \H^1 \hookrightarrow \L^4 $, we have that for any $ t \in [0,T] $ and $ p \in [1,\infty) $, 
	\begin{align*}
		&\E \left[ \int_0^t |p_\epsilon(m_\epsilon)|_{\L^2}^{2p} \ds \right] \\
		&\lesssim \E \left[\int_0^t \la 
		|M_\epsilon|_{\L^2_\rho} + |M_\epsilon|_{\L^\infty} |M_\epsilon|_{\L^4_\rho}^2 
		+ (1+|u_\epsilon|_{\L^\infty}) |m_\epsilon|_{\L^2} 
		+ |M_\epsilon|_{\L^\infty} \la |m_\epsilon|_{\L^\infty} |m_\epsilon + u_\epsilon|_{\L^2} + |m_\epsilon|_{\L^2} \ra
		\ra^{2p} \ds \right] \\
		&\lesssim \E \left[\int_0^t \la (1+|m_\epsilon|_{\H^2}) |m_\epsilon+e_3|_{\H^1_\rho}^2 
		+ (1+|m_\epsilon|_{\H^2}) |m_\epsilon|_{\L^2} 
		+ (1+|m_\epsilon|_{\H^2}) \la |m_\epsilon|_{\H^2}^2 + |m_\epsilon|_{\L^2} \ra \ra^{2p} \ds \right] \\
		&\leq c(R). 
	\end{align*}
	In addition, 
	\begin{equation}\label{eq: |d(f)|_Lprho}
		|d_\epsilon(f)|_{\L^4_\rho}^4 
		\lesssim |f|_{\L^\infty}^2 |d_\epsilon(f)|^2_{\L^2_\rho} 
		\lesssim |f|_{\L^\infty}^2 |d_\epsilon(f)|_{\L^2} |d_\epsilon(f)|_{\L^2_\rho} 
		\lesssim |f|_{\H^2}^3 |d_\epsilon(f)|_{\L^2_\rho}. 
	\end{equation}
	For remainders of higher-order terms, 
	\begin{align*}
		(1-|u_\epsilon+e_3|^2)^2 
		&= (1-|M_\epsilon + d_\epsilon(m_\epsilon)|^2)^2 \\
		&= (1-|M_\epsilon|^2)^2 
		+ \la |d_\epsilon (m_\epsilon)|^2 + 2\langle d_\epsilon(m_\epsilon), M_\epsilon \rangle \ra^2 \\
		&\quad + 2 (|m_\epsilon|^2 + 2 \lb m_\epsilon, e_3 \rb) \la |d_\epsilon(m_\epsilon)|^2 + 2\langle d_\epsilon(m_\epsilon), M_\epsilon \rangle \ra,
	\end{align*}
	Thus,
	\begin{equation}\label{eq: Beps part 1}
		\begin{aligned}
			\left|1-|u_\epsilon+e_3|^2 \right|_{\L^2_\rho}^2 
			&\lesssim
			\left| 1-|M_\epsilon|^2 \right|_{\L^2_\rho}^2 
			+ |d_\epsilon(m_\epsilon)|_{\L^4_\rho}^4 \\
			&\quad + \la |M_\epsilon|_{\L^\infty}^2 + |m_\epsilon|_{\L^\infty}^2 + |m_\epsilon|_{\L^\infty} \ra |d_\epsilon(m_\epsilon)|_{\L^2_\rho}^2 \\
			&\quad + (|m_\epsilon|_{\L^4}^2 + |m_\epsilon|_{\L^2}) |M_\epsilon|_{\L^\infty} |d_\epsilon(m_\epsilon)|_{\L^2_\rho} \\
			&\lesssim
			\left| 1-|M_\epsilon|^2 \right|_{\L^2_\rho}^2 
			+ (1+|m_\epsilon|_{\H^2}^3) |d_\epsilon(m_\epsilon)|_{\L^2_\rho}, 
		\end{aligned}
	\end{equation}
	where the last inequality holds by \eqref{eq: |d(f)|_Lprho} and \eqref{eq: |Ju|_L2rho <= c|u|_L2rho}.

	\begin{lemma}\label{Lemma: |meps+e3| <= remainder}
		There exist constants $ c_1 = c_1(R,T)>0 $ and $ c_2 = c_2(T) > 0 $, both independent of $ \epsilon $, such that
		\begin{align*}
			\E \left[ |1-|m_\epsilon(t)+e_3|^2|_{\L^2_\rho}^2 \right] 
			&\leq 
			c_1 \ e^{c_2} \la |1-|J_\epsilon m_0+e_3|^2|_{\L^2_\rho}^2 
			+ \E \left[ \int_0^t A_\epsilon(m_\epsilon(s)) \ds \right]^\frac{1}{2} \ra,
		\end{align*}
		for any $ t \in [0,T] $, where
		\begin{equation}\label{eq: defn Aeps}
			A_\epsilon(m_\epsilon) := |d_\epsilon(m_\epsilon)|_{\H^1_\rho}^2 + \sum_k |d_\epsilon(\nabla_{g_k} J_\epsilon m_\epsilon)|^2_{\L^2_\rho} + |p_\epsilon(m_\epsilon)|_{\L^2}^2. 
		\end{equation}
	\end{lemma}
	\begin{proof}	 
		Recall $ \varphi(m_\epsilon) $ from \eqref{eq: defn varphi(meps)}. 
		We have $ \varphi(m_\epsilon) \leq c (|m_\epsilon|_{\L^4}^4 + |m_\epsilon|_{\L^2}^2) $ and
		\begin{align*}
			\varphi'(m_\epsilon) u &= -\lb (1-|M_\epsilon|^2) M_\epsilon, u \rb_{\L^2_\rho}, \\
			\varphi''(m_\epsilon)(u,w) &= -\lb (1-|M_\epsilon|^2) u,w \rb_{\L^2_\rho}+ 2 \int_{\R^d} \lb M_\epsilon, u\rb \lb M_\epsilon, w \rb \rho(x) \dx.
		\end{align*}	
		By It{\^o}'s formula, 
		\begin{equation}\label{eq: |M|=1 pf Ito}
			\begin{aligned}
				\frac{1}{4}\d|1-|M_\epsilon(t)|^2|_{\L^2_\rho}^2
				&= -\lb (1-|M_\epsilon|^2) M_\epsilon, F^R_\epsilon(M_\epsilon) \rb_{\L^2_\rho} \dt \\
				&\quad - \frac{1}{2} \sum_k \lb (1-|M_\epsilon|^2) M_\epsilon, S_{k,\epsilon}(m_\epsilon) \rb_{\L^2_\rho} \dt \\
				&\quad - \frac{1}{2} \sum_k \lb (1-|M_\epsilon|^2) G_{k,\epsilon}(m_\epsilon),G_{k,\epsilon}(m_\epsilon) \rb_{\L^2_\rho} \dt \\
				&\quad + \sum_k \int_{\R^d} \lb M_\epsilon, G_{k,\epsilon}(m_\epsilon) \rb^2 \rho(x) \dx \dt \\
				&\quad - \sum_k \lb (1-|M_\epsilon|^2) M_\epsilon, G_{k,\epsilon}(m_\epsilon) \rb_{\L^2_\rho} \dW_k \\
				&=: {\rm U}_1 \dt - \frac{1}{2} \sum_k \la {\rm U}_{2,k} + {\rm U}_{3,k} - 2{\rm U}_{4,k} \ra \dt - \sum_k {\rm U}_{5,k} \dW_k.
			\end{aligned}
		\end{equation}
		
		For the diffusion term $ {\rm U}_{5,k} $, 	
		\begin{align*}
			{\rm U}_{5,k}
			&= -\lb (1-|M_\epsilon|^2) M_\epsilon, G_{k,\epsilon}(m_\epsilon) \rb_{\L^2_\rho} \\
			&\lesssim |1-|M_\epsilon|^2|_{\L^2_\rho} |M_\epsilon|_{\L^\infty} |\nabla m_\epsilon|_{\L^2_\rho} \\
			&\lesssim \la |m_\epsilon|_{\L^4}^4 + |m_\epsilon|_{\L^2}^2 \ra \la 1+ |m_\epsilon|_{\H^2} \ra |\nabla m_\epsilon|_{\L^2}.
		\end{align*}
		Since $ m_\epsilon \in L^{2p}(\Omega;L^\infty(0,T;\H^2)) $ for $ p \in [1,\infty) $, $ {\rm U}_{5,k} \in L^{2p}(\Omega;L^2(0,T)) $. Then the It{\^o} integral is well-defined, thus $ \sum_k {\rm U}_{5,k} \dW_k $ has zero expectation. 
		
		For $ {\rm U}_1 $, we have
		\begin{align*}
			{\rm U}_1
			&= -\lb J_\epsilon \la (1-|M_\epsilon|^2) M_\epsilon \rho \ra, \psi^R_\epsilon \bar{F}(u_\epsilon+e_3) - \nabla_v u_\epsilon \rb_{\L^2} \\
			&= -\lb (1- |u_\epsilon+e_3|^2) (u_\epsilon+e_3), \psi^R_\epsilon \bar{F}(u_\epsilon+e_3) \rb_{\L^2_\rho} \\
			&\quad + \lb (1- |u_\epsilon+e_3|^2) (u_\epsilon+e_3), \nabla_v (u_\epsilon+e_3) \rb_{\L^2_\rho} \\
			&\quad - \lb p_\epsilon(m_\epsilon), \psi^R_\epsilon \bar{F}(u_\epsilon+e_3) - \nabla_v u_\epsilon \rb_{\L^2} \\
			&=: {\rm U}_{1a} + {\rm U}_{1b} + {\rm U}_{1c}, 
		\end{align*}
		where $ {\rm U}_{1a} $ is equal to $ 0 $ by the cross-product structure of $ \bar{F} $. 
		For $ {\rm U}_{1b} $, since $ {\rm div} (v) = 0 $ and $ v \rho $ vanishes at infinity, we have $ {\rm div}(v \rho) = \langle v, \nabla \rho \rangle $ and $ \int_{\R^d} {\rm div}(v \rho) \dx = 0 $. 
		Then by \eqref{eq: <(1-|u|^2)u,Dfu>}, 
		\begin{align*}
			{\rm U}_{1b}
			&= \frac{1}{4} \int_{\R^d} {\rm div}(v \rho) (1-|u_\epsilon+e_3|^2)^2 \dx - \frac{1}{4} \int_{\R^d} {\rm div} (v \rho) \dx \\
			&= \frac{1}{4} \int_{\R^d} \lb v, \rho^{-1} \nabla \rho \rb (1-|u_\epsilon+e_3|^2)^2 \rho \dx \\
			&\lesssim |v|_{\L^\infty} |\rho^{-1}\nabla \rho|_{\L^\infty} \left|1-|u_\epsilon+e_3|^2 \right|_{\L^2_{\rho}}^2, 
		\end{align*} 
		where $ |\rho^{-1} \nabla \rho|_{\L^\infty} = \frac{1}{2} $. 
		Thus, by \eqref{eq: Beps part 1} and the estimates in Lemma \ref{Lemma: meps H2-est}, for any $ t \in [0,T] $, 
		\begin{align*}
			\E \left[\int_0^t {\rm U}_1 \ds \right]
			&\lesssim 
			\E \left[ \int_0^t \left| 1-|M_\epsilon|^2 \right|_{\L^2_\rho}^2 \ds 
			+ \int_0^t (1+|m_\epsilon|_{\H^2}^3) |d_\epsilon(m_\epsilon)|_{\L^2_\rho} \ds \right] \\
			&\quad + \E \left[\int_0^t |p_\epsilon(m_\epsilon)|_{\L^2} |\psi^R_\epsilon \bar{F}(u_\epsilon+e_3) - \nabla_v u_\epsilon|_{\L^2} \ds \right] \\
			&\leq 
			c \E \left[ \int_0^t \left| 1-|M_\epsilon|^2 \right|_{\L^2_\rho}^2 \ds \right] 
			+ c(R) \ \E \left[ \int_0^t \la |d_\epsilon(m_\epsilon)|^2_{\L^2_\rho} + |p_\epsilon(m_\epsilon)|_{\L^2}^2 \ra \ds \right]^\frac{1}{2}. 
		\end{align*}
		
		For $ {\rm U}_{2,k} $,  
		\begin{align*}
			{\rm U}_{2,k}
			&= \lb J_\epsilon \la (1-|M_\epsilon|^2) M_\epsilon \rho \ra, S_k(u_\epsilon) \rb_{\L^2} \\
			&= \lb (1-|u_\epsilon+e_3|^2)(u_\epsilon+e_3), (\nabla_{g_k})^2 u_\epsilon \rb_{\L^2_\rho} 
			+ \lb p_\epsilon(m_\epsilon), (\nabla_{g_k})^2 u_\epsilon \rb_{\L^2} \\
			&=: {\rm U}_{2a,k} + {\rm U}_{2b,k},
		\end{align*}
		where by \eqref{eq: g in W2inf} and the estimates in Lemma \ref{Lemma: meps H2-est}, 
		\begin{align*}
			\E \left[ \sum_k \int_0^t |{\rm U}_{2b,k}| \ds \right] 
			&\leq \E\left[\sum_k |g_k|_{W^{1,\infty}(\R^d;\R^d)}^2 \int_0^t |p_\epsilon(m_\epsilon)|_{\L^2} |\nabla u_\epsilon|_{\H^1} \ds \right] \\
			&\leq c(R) \ \E\left[ \int_0^t |p_\epsilon(m_\epsilon)|_{\L^2}^2 \ds \right]^\frac{1}{2}. 
		\end{align*}
		For $ {\rm U}_{2a,k} $, integrating-by-parts,
		\begin{align*}
			{\rm U}_{2a,k}
			&= \lb (1-|u_\epsilon+e_3|^2) (u_\epsilon+e_3), (\nabla_{g_k})^2 u_\epsilon \rb_{\L^2_\rho} \\
			&= \lb (1-|u_\epsilon+e_3|^2) (u_\epsilon+e_3), \nabla_{\rho g_k} \nabla_{g_k} u_\epsilon \rb_{\L^2} \\
			&= -\int_{\R^d} (1-|u_\epsilon+e_3|^2) |\nabla_{g_k} u_\epsilon|^2 \rho \dx  
			+ 2\int_{\R^d} \lb u_\epsilon+e_3, \nabla_{g_k} u_\epsilon \rb^2 \rho \dx  \\
			&\quad - \lb {\rm div} (g_k \rho) (1-|u_\epsilon+e_3|^2)(u_\epsilon+e_3), \nabla_{g_k} (u_\epsilon+e_3) \rb_{\L^2} \\
			&=: -\widehat{{\rm U}}_{3,k} + 2\widehat{{\rm U}}_{4,k} + {\rm V}_{2a,k}, 
		\end{align*}
		where $ -\widehat{{\rm U}}_{3,k} + 2\widehat{{\rm U}}_{4,k} $ is in a similar form to $ -{\rm U}_{3,k} + 2{\rm U}_{4,k} $. 
		For the remainder $ {\rm V}_{2a,k} $, by \eqref{eq: <(1-|u|^2)u,Dfu>}, 
		\begin{align*}
			{\rm V}_{2a,k}
			&= - \frac{1}{4} \la \int_{\R^d} {\rm div}\la {\rm div}(g_k \rho) g_k \ra\rho^{-1}  \left| 1-|u_\epsilon+e_3|^2 \right|^2 \rho \dx - \int_{\R^d} {\rm div}\la {\rm div}(g_k \rho) g_k \ra \dx \ra, 
		\end{align*}
		where the fact $ \rho $ and $ \nabla \rho $ vanish at infinity implies $ \int_{\R^d} {\rm div}({\rm div}(g_k \rho) g_k) \dx = 0 $, 
		and by \eqref{eq: rho^-1 D-D2},
		\begin{align*}
			{\rm div}\la {\rm div}(g_k \rho) g_k \ra\rho^{-1}
			&= ({\rm div} g_k)^2 + \lb \nabla_{g_k} g_k + 2g_k {\rm div} g_k, \rho^{-1}\nabla \rho \rb + \nabla_{g_k} {\rm div} g_k + \lb g_k, \rho^{-1} \nabla_{g_k} \nabla \rho \rb \\
			&\lesssim |g_k|^2_{W^{2,\infty}(\R^d;\R^d)}.
		\end{align*}
		Then by \eqref{eq: g in W2inf} and \eqref{eq: Beps part 1}, we are left with
		\begin{align*}
			\E \left[ \int_0^t \sum_k |{\rm V}_{2a,k}| \ds \right]
			&\leq 
			c \E \left[ \int_0^t |1-|M_\epsilon(s)|^2|_{\L^2_\rho}^2 \ds \right] 
			+ c(R) \E \left[ \int_0^t |d_\epsilon(m_\epsilon)|^2_{\L^2_\rho} \ds \right]^\frac{1}{2}. 
		\end{align*}
		
		Now we estimate $ {\rm U}_{3,k} $ and $ {\rm U}_{4,k} $. 
		Note that $ G_{k,\epsilon}(m_\epsilon)-\nabla_{g_k} u_\epsilon = d_\epsilon(\nabla_{g_k} u_\epsilon) $. 
		We have
		\begin{align*}
			{\rm U}_{3,k} 
			&= \lb (1-|M_\epsilon|^2) G_{k,\epsilon}(m_\epsilon), G_{k,\epsilon}(m_\epsilon) \rb_{\L^2_\rho} \\
			&= \lb (1-|u_\epsilon + e_3|^2) \nabla_{g_k} u_\epsilon, \nabla_{g_k} u_\epsilon \rb_{\L^2_\rho} \\
			&\quad + \lb \lb d_\epsilon(m_\epsilon), m_\epsilon+ u_\epsilon+2e_3 \rb G_{k,\epsilon}(m_\epsilon), G_{k,\epsilon}(m_\epsilon) \rb_{\L^2_\rho} \\
			&\quad + \lb (1-|u_\epsilon + e_3|^2) d_\epsilon(\nabla_{g_k} u_\epsilon), G_{k,\epsilon}(m_\epsilon) + \nabla_{g_k} u_\epsilon \rb_{\L^2_\rho} \\
			&= \widehat{{\rm U}}_{3,k} + {\rm V}_{3,k},
		\end{align*}
		and thus
		\begin{align*}
			\E \left[\sum_k \int_0^t |{\rm V}_{3,k}| \ds \right] 
			&\lesssim 
			\E\left[\sum_k |g_k|_{L^\infty(\R^d;\R^d)}^2 \int_0^t |d_\epsilon(m_\epsilon)|_{\L^2_\rho} \la |m_\epsilon|_{\L^\infty}+1 \ra |\nabla m_\epsilon|_{\L^4}^2 \ds \right] \\
			&\quad + \E\left[ \sum_k |g_k|_{L^\infty(\R^d;\R^d)} \int_0^t |d_\epsilon(\nabla_{g_k} u_\epsilon)|_{\L^2_\rho} \la |m_\epsilon|_{\L^\infty}^2 + |m_\epsilon|_{\L^\infty} \ra |\nabla m_\epsilon|_{\L^2} \ds \right] \\
			&\leq c(R) \E \left[ \int_0^t \la |d_\epsilon(m_\epsilon)|^2_{\L^2_\rho} + \sum_k |d_\epsilon(\nabla_{g_k} u_\epsilon)|_{\L^2_\rho}^2 \ra \ds \right]^\frac{1}{2}.
		\end{align*}
		Similarly, 
		\begin{align*}
			\lb M_\epsilon, G_{k,\epsilon}(m_\epsilon) \rb - \lb u_\epsilon+e_3, \nabla_{g_k} u_\epsilon \rb 
			&= \lb u_\epsilon+e_3, d_\epsilon(\nabla_{g_k} u_\epsilon) \rb - \lb d_\epsilon(m_\epsilon), G_{k,\epsilon}(m_\epsilon) \rb
			=: r_{4,k}.
		\end{align*}
		By \eqref{eq: |d(f)|_Lprho} and \eqref{eq: |Ju|_L2rho <= c|u|_L2rho}, 
		\begin{align*}
			|r_{4,k}|_{\L^2_\rho}^2
			&\lesssim (1+|m_\epsilon|_{\L^\infty})^2 |d_\epsilon(\nabla_{g_k} u_\epsilon)|_{\L^2_\rho}^2 + |g_k|^2_{L^\infty(\R^d;\R^d)} |\nabla m_\epsilon|_{\L^4}^2  |d_\epsilon(m_\epsilon) \rho^\frac{1}{4}|_{\L^4_\rho}^2 \\
			&\lesssim 
			\la |g_k|_{L^\infty(\R^d;\R^d)} (1+|m_\epsilon|_{\H^2}^3) + |g_k|^2_{L^\infty(\R^d;\R^d)} |m_\epsilon|_{\H^2}^3 \ra \la |d_\epsilon(m_\epsilon)|_{\L^2_\rho} + |d_\epsilon(\nabla_{g_k} u_\epsilon)|_{\L^2_\rho} \ra,
		\end{align*}
		Then for $ {\rm U}_{4,k} $, we have
		\begin{align*}
			{\rm U}_{4,k}
			&= \int_{\R^d} \lb M_\epsilon, G_{k,\epsilon}(m_\epsilon) \rb^2 \rho \dx \\
			&= \int_{\R^d} \lb u_\epsilon+e_3, \nabla_{g_k} u_\epsilon \rb^2 \rho \dx + \int_{\R^d} \la r_{4,k}^2 + 2 \lb u_\epsilon+e_3, \nabla_{g_k} u_\epsilon \rb r_{4,k} \ra \rho \dx \\
			&= \widehat{{\rm U}}_{4,k} + {\rm V}_{4,k}, 
		\end{align*}
		where 
		\begin{align*}
			\E \left[\sum_k \int_0^t |{\rm V}_{4,k}| \ds \right] 
			&\lesssim 
			\E\left[\sum_k \int_0^t |r_{4,k}|_{\L^2_\rho}^2 + 2 |g_k|_{\L^\infty(\R^d;\R^d)} (1+|m_\epsilon|_{\L^\infty}) |\nabla m_\epsilon|_{\L^2} |r_{4,k}|_{\L^2_\rho} \ds \right] \\
			&\leq c(R) \E \left[ \int_0^t \la |d_\epsilon(m_\epsilon)|^2_{\L^2_\rho} + |d_\epsilon(m_\epsilon) \rho^\frac{1}{4}|_{\L^4_\rho}^2 + \sum_k |d_\epsilon(\nabla_{g_k} u_\epsilon)|^2_{\L^2_\rho} \ra \ds \right]^\frac{1}{2}.
		\end{align*}
		Note that
		\begin{align*}
			|d_\epsilon(m_\epsilon) \rho^\frac{1}{4}|_{\L^4_\rho}^2
			&= |d_\epsilon(m_\epsilon) \rho^\frac{1}{2}|_{\L^4}^2 
			\lesssim |d_\epsilon(m_\epsilon) \rho^\frac{1}{2}|_{\H^1}^2
			= |d_\epsilon(m_\epsilon)|_{\H^1_\rho}^2.
		\end{align*}
		Hence, the main components of the Stratonovich and It{\^o} corrections cancel with each other, leaving 
		\begin{align*}
			-\E \left[ \sum_k \la {\rm U}_{2,k} + {\rm U}_{3,k} -2{\rm U}_{4,k} \ra \dt \right]
			&\leq 
			c\E \left[ \int_0^t |1-|M_\epsilon(s)|^2|_{\L^2_\rho}^2 \ds \right] \\
			&\quad + c(R) \E\left[ \int_0^t \la |d_\epsilon(m_\epsilon)|_{\H^1_\rho}^2 + \sum_k |d_\epsilon(\nabla_{g_k} u_\epsilon)|^2_{\L^2_\rho} + |p_\epsilon(m_\epsilon)|_{\L^2}^2 \ra {\rm d}s \right]^\frac{1}{2},  
		\end{align*}

		Overall, we obtain
		\begin{align*}
			\E \left[ |1-|M_\epsilon(t)|^2|_{\L^2_\rho}^2 \right] 
			&\leq 
			|1-|J_\epsilon m_0+e_3|^2|_{\L^2_\rho}^2 + c \E \left[ \int_0^t |1-|M_\epsilon(s)|^2|_{\L^2_\rho}^2 \ds \right] \\
			&\quad + c(R) \E \left[ \int_0^t A_\epsilon(m_\epsilon) \ds \right]^\frac{1}{2}, 
		\end{align*}
		for any $ t \in [0,T] $. 
		Then we apply Gronwall's lemma to yield the desired result. 
	\end{proof}

\section{Convergences for fixed $ R $}\label{Section: convergence for fixed R}
	In this section, we fix $ R > 1 $, and employ Skorohod theorem and compactness argument to deduce convergences of (functions of) new approximations $ \tilde{m}_\epsilon $. 
	These convergences and subsequent regularity properties of the limit $ \tilde{m} $ allow us to show that $ \tilde{m}+e_3 $ is a martingale solution of \eqref{eq: sLLG with Delta^2} later in Section \ref{Section: Proof of Theorem}. 
	
	For Banach spaces $ B_0, B_1, B_2 $ and $ B_3 $ such that $ B_0, B_2 $ are reflexive and $ B_0 \Subset B_1 \hookrightarrow B_2 \Subset B_3 $, the following embeddings follow from \cite[Theorems 2.1 and 2.2]{FlandoliGatarek}:
	\begin{align*}
		L^p(0,T;B_0) \cap W^{\sigma,p}(0,T;B_2) &\Subset L^p(0,T;B_1), \\
		W^{\sigma,p}(0,T; B_2) &\Subset \mathcal{C}([0,T];B_3), 
	\end{align*}
	for $ p \in (1,\infty), \sigma \in (0,1) $ and $ \sigma p > 1 $. 
	We define
	\begin{align*}
		E_0 &:= L^\infty(0,T;\H^2) \cap W^{\sigma,p}(0,T;\L^2), \\
		E &:= L^2(0,T; \H^1_\rho) \cap \mathcal{C}([0,T]; \H^{-1}),
	\end{align*}
	for $ p \in [2,\infty) $, $ \sigma \in (0,\frac{1}{2}) $, $ \sigma- \frac{1}{p} < \frac{1}{2} $ and $ \sigma p >1 $. 
	Since $ \L^2 \hookrightarrow \L^2_\rho $, $ \H^2 \Subset \H^1_\rho \hookrightarrow \L^2_\rho $ and $ \H^1 \Subset \L^2_\rho = (\L^2_\rho)^* \Subset \H^{-1} $, we have 
	\begin{align*}
		E_0 
		\hookrightarrow 
		L^p(0,T;\H^2) \cap W^{\sigma,p}(0,T;\L^2_\rho) 
		\Subset
		E. 
	\end{align*}
	By Lemma \ref{Lemma: meps H2-est}, $ m_\epsilon $ takes values in $ E_0 $, $ \P $-a.s. for any $ \epsilon > 0 $.  
	Hence, the set of laws $ \{ \mathcal{L}(m_\epsilon) \} $ on $ E $ is tight, where for any $ \lambda > 0 $, 
	$ \{ |m_\epsilon|_{E_0} \leq \lambda \} $ is compact in $ E $ and 
	\begin{align*}
		\P \la |m_\epsilon|_{E_0} > \lambda \ra 
		\leq \lambda^{-2} \E \left[ |m_\epsilon|_{E_0}^2 \right] 
		\leq \lambda^{-2} c(R) 
		\to 0, \quad \text{ as } \lambda \to \infty.
	\end{align*}
	
	Since $ E $ is a separable metric space, by Skorohod theorem there exists a probability space $ (\tilde{\Omega}, \tilde{\mathcal{F}}, \tilde{\P}) $ and a sequence $ \{(\tilde{m}_\epsilon, \tilde{W}_\epsilon)\} $ of $ E \times \mathcal{C}([0,T];H^4(\R^d;\R^d)) $-valued random variables defined on $ (\tilde{\Omega}, \tilde{\mathcal{F}}, \tilde{\P}) $ such that 
	\begin{align*}
		\mathcal{L}((m_\epsilon, W)) = \mathcal{L}((\tilde{m}_\epsilon, \tilde{W}_\epsilon)) \quad \text{on } E \times \mathcal{C}([0,T];H^4(\R^d;\R^d)), \quad \forall \epsilon > 0,
	\end{align*}
	and there exists an $ E \times \mathcal{C}([0,T];H^4(\R^d;\R^d)) $-valued random variable $ (\tilde{m}, \tilde{W}) $ such that
	\begin{equation}\label{eq: meps to m, Weps to W ptw}
		\tilde{m}_\epsilon \to \tilde{m} \ \text{ in } E, \quad
		\tilde{W}_\epsilon \to \tilde{W} \ \text{ in } \mathcal{C}([0,T];H^4(\R^d;\R^d)), \quad \tilde{\P}\text{-a.s.}
	\end{equation}
	The space $ \mathcal{C}([0,T];\H^2) $ is separable and complete. 
	Then by Kuratowski's theorem, $ \{\tilde{m}_\epsilon\} $ have the same laws as $ \{m_\epsilon\} $ on $ \mathcal{C}([0,T];\H^2) $, which implies that the estimates in Lemma \ref{Lemma: meps H2-est} also hold for $ \{\tilde{m}_\epsilon \} $. 
	
	We can extend the definition of the $ \H^2 $-norm to $ \H^{-1} $ by setting $ |u|_{\H^2} = \infty $ for $ \smash{u \in \H^{-1} \setminus \H^2} $. Then using the pointwise convergence \eqref{eq: meps to m, Weps to W ptw} in $ \smash{\mathcal{C}([0,T];\H^{-1})} $ and the uniform integrability of $ \tilde{m}_\epsilon $ in $ \smash{L^{2p}(\tilde{\Omega}; \mathcal{C}([0,T];\H^2))} $, we deduce that 
	\begin{equation}\label{eq: m in Linft-H2}
		\E \left[ |\tilde{m}|_{L^\infty(0,T;\H^2)}^{2p} \right] < \infty, \quad \forall p \in [1,\infty).
	\end{equation}	
	As a result, by \eqref{eq: meps to m, Weps to W ptw}, \eqref{eq: |Ju|_L2rho <= c|u|_L2rho} and \eqref{eq: <Ju,w>_L2rho}, for $ p \in [1,\infty) $,   
	\begin{equation}\label{eq: meps to m}
		\begin{alignedat}{2}
			&\tilde{m}_\epsilon \to \tilde{m} \quad \text{ in } L^{2p}(\tilde{\Omega}; L^2(0,T;\H^1_{\rho})), \quad
			&&\Delta \tilde{m}_\epsilon \rightharpoonup \Delta \tilde{m} \quad \text{ in } L^{2p}(\tilde{\Omega}; L^2(0,T;\L^2_{\rho})), \\
			&J_\epsilon \tilde{m}_\epsilon \to \tilde{m} \quad \text{ in } L^{2p}(\tilde{\Omega}; L^2(0,T; \H^1_{\rho})), \quad
			&&\Delta J_\epsilon \tilde{m}_\epsilon \rightharpoonup \Delta \tilde{m} \quad \text{ in } L^{2p}(\tilde{\Omega}; L^2(0,T; \L^2_{\rho})).
		\end{alignedat}
	\end{equation} 
	By Gagliardo-Nirenberg's inequality,
	\begin{equation}\label{eq: |Jmeps-m| Lprho}
		\left|(\tilde{m}_\epsilon - \tilde{m}) \rho^{\frac{1}{2}} \right|_{\L^r}
		\lesssim \left|\tilde{m}_\epsilon - \tilde{m} \right|_{\L^2_\rho}^{1-\theta} \la |\tilde{m}_\epsilon|^\theta_{\H^2} + |\tilde{m}|^\theta_{\H^2} \ra, \quad r \in [2,\infty], 
	\end{equation}
	where $ \theta = \frac{d}{4}-\frac{d}{2r} \in [0,1) $ for $ d = 2,3 $. 
	The inequality \eqref{eq: |Jmeps-m| Lprho} still holds with $ J_\epsilon \tilde{m}_\epsilon $ in place of $ \tilde{m}_\epsilon $. 
	Then by the $ L^\infty(0,T;\H^2) $ estimates of $ \tilde{m}_\epsilon $ and $ \tilde{m} $, 
	\begin{equation}\label{eq: Jmeps to m in Lp, Linf}
		\begin{aligned}
			\tilde{m}_\epsilon \rho^\frac{1}{2} &\to \tilde{m} \rho^\frac{1}{2} \quad \text{ in } L^{2p}(\tilde{\Omega}; L^{q}(0,T;\L^r)) \\
			(J_\epsilon \tilde{m}_\epsilon) \rho^\frac{1}{2} &\to \tilde{m} \rho^\frac{1}{2} \quad \text{ in } L^{2p}(\tilde{\Omega}; L^{q}(0,T;\L^r)), 
		\end{aligned}
	\end{equation} 
	for $ p \in [1,\infty) $, $ q \leq \frac{2}{1-\theta} $ and $ r \in [2,\infty] $.

\subsection{Vector length of $ \tilde{m} $}\label{Section: vector length of m-tilde}
	As in Lemma \ref{Lemma: |meps+e3| <= remainder}, let $ \varphi(\tilde{m}) := |1-|\tilde{m}+e_3|^2|_{\L^2_\rho} \lesssim |\tilde{m}|_{\L^4}^2 + |\tilde{m}|_{\L^2}^2 \in L^{2p}(\tilde{\Omega}; L^\infty(0,T)) $. 
	Then
	\begin{equation}\label{eq: |tilde m| vs |meps|}
		\begin{aligned}
			&\tilde{\E} \left[ \int_0^T |1-|\tilde{m}(t)+e_3|^2|_{\L^2_\rho}^2 \dt \right] \\
			&\lesssim \tilde{\E} \left[ \int_0^T \la \left|1-|\tilde{m}_\epsilon(t)+e_3|^2 \right|_{\L^2_\rho}^2 + \left||\tilde{m}_\epsilon(t)+e_3|^2 - |\tilde{m}(t)+e_3|^2 \right|_{\L^2_\rho}^2 \ra \dt \right] \\
			&\lesssim \tilde{\E}\left[ \int_0^T \la \left|1-|\tilde{m}_\epsilon(t)+e_3|^2 \right|_{\L^2_\rho}^2 + |\tilde{m}_\epsilon(t) - \tilde{m}(t)|_{\L^4_\rho}^2 |\tilde{m}_\epsilon(t) + \tilde{m}(t)|_{\L^4_\rho}^2 + |\tilde{m}_\epsilon(t) - \tilde{m}(t)|_{\L^2_\rho}^2 \ra \dt \right] \\
			&\lesssim \tilde{\E}\left[ \int_0^T \left|1-|\tilde{m}_\epsilon(t)+e_3|^2 \right|_{\L^2_\rho}^2 \dt \right] 
			+ \tilde{\E} \left[ \sup_{t \in [0,T]} \la |\tilde{m}_\epsilon(t) + \tilde{m}(t)|_{\H^1_\rho}^2+1 \ra \int_0^T |\tilde{m}_\epsilon(t) - \tilde{m}(t)|_{\H^1_\rho}^2 \dt \right],
		\end{aligned}
	\end{equation}
	where the second expectation on the right-hand side converges to $ 0 $ by \eqref{eq: m in Linft-H2} and \eqref{eq: meps to m}, thus we only need to focus on the first expectation.  
	Recall the definitions \eqref{eq: defn peps, qeps} and \eqref{eq: defn Aeps}. 
	For any $ q \in [4,\infty) $, the following maps are measurable:
	\begin{align*}
		L^q(0,T;\H^2) \ni u &\mapsto |1-|u+e_3|^2|_{\L^2_\rho}^2 \in L^1(0,T), \\
		L^q(0,T;\H^2) \ni u &\mapsto A_\epsilon(u) \in L^1(0,T). 
	\end{align*}
	Then since $ \{ \tilde{m}_\epsilon \} $ has the same laws as $ \{ m_\epsilon \} $ on $ L^q(0,T;\H^2) $, it holds by Lemma \ref{Lemma: |meps+e3| <= remainder} that
	\begin{align*}
		\tilde{\E} \left[ \int_0^T |1-|\tilde{m}_\epsilon(t)+e_3|^2|_{\L^2_\rho}^2 \dt \right] 
		&= \E \left[ \int_0^T |1-|m_\epsilon(t)+e_3|^2|_{\L^2_\rho}^2 \dt \right] \\
		&\lesssim |1-|J_\epsilon m_0+e_3|^2|_{\L^2_\rho}^2 
		+ \tilde{\E} \left[ \int_0^t A_\epsilon(\tilde{m}_\epsilon(s)) \ds \right]^\frac{1}{2}.
	\end{align*} 
	By the assumption that $ m_0 $ is bounded in $ \H^2 $ and the approximation property of $ J_\epsilon $, 
	\begin{align*}
		|1-|J_\epsilon m_0 + e_3|^2|_{\L^2_\rho}^2
		&= \left| |m_0+e_3|^2 - |J_\epsilon m_0 + e_3|^2 \right|_{\L^2_\rho}^2 \\
		&\leq |J_\epsilon m_0 - m_0|_{\L^2_\rho}^2 |J_\epsilon m_0 + m_0 + 2e_3|_{\L^\infty}^2 \\
		&\lesssim \la 1+|m_0|_{\H^2}^2 \ra |J_\epsilon m_0 - m_0|_{\L^2}^2 \  
		\to 0, \quad \text{ as } \epsilon \to 0. 
	\end{align*}
	For the remainder $ A_\epsilon(\tilde{m}_\epsilon) $, we first observe that by \eqref{eq: meps to m}, 
	\begin{align*}
		d_\epsilon(\tilde{m}_\epsilon)
		&= (J_\epsilon \tilde{m}_\epsilon - \tilde{m}) + (\tilde{m}- \tilde{m}_\epsilon) 	
		\to 0 \quad \text{in } L^{2p}(\tilde{\Omega};L^2(0,T;\H^1_\rho)).
	\end{align*}
	Similarly, 
	\begin{align*}
		\sum_k |d_\epsilon(\nabla_{g_k} J_\epsilon \tilde{m}_\epsilon) |_{\L^2_\rho}^2
		&= \sum_k \left|J_\epsilon \nabla_{g_k} \la J_\epsilon \tilde{m}_\epsilon - \tilde{m} \ra
		+ \la J_\epsilon \nabla_{g_k} \tilde{m} - \nabla_{g_k} \tilde{m} \ra
		+ \nabla_{g_k} \la \tilde{m} - J_\epsilon \tilde{m}_\epsilon \ra \right|_{\L^2_\rho}^2 \\
		&\lesssim \sum_k |g_k|^2_{L^\infty(\R^d;\R^d)} |\nabla (J_\epsilon \tilde{m}_\epsilon - \tilde{m})|_{\L^2_\rho}^2
		+ \sum_k \left| J_\epsilon \nabla_{g_k}\tilde{m} - \nabla_{g_k} \tilde{m} \right|_{\L^2_\rho}^2,
	\end{align*}
	where the right-hand side converges to $ 0 $ in $ L^{2p}(\tilde{\Omega};L^1(0,T)) $. 
	Also, by the $ L^\infty(0,T;\H^2) $-estimate in \eqref{eq: m in Linft-H2}, the continuous embedding $ \H^2 \hookrightarrow \L^\infty $ and Lemma \ref{Lemma: meps H2-est}, 
	\begin{align*}
		&\tilde{\E} \left[ \int_0^T \left| (1-|\tilde{m}_\epsilon+e_3|^2) (\tilde{m}_\epsilon+e_3) \rho - (1-|\tilde{m}+e_3|^2) (\tilde{m}+e_3) \rho \right|_{\L^2}^2 \dt \right] \\
		&\lesssim \tilde{\E} \left[ \int_0^T \la \left| \tilde{m}_\epsilon-\tilde{m} \right|_{\L^2_\rho}^2 + \left| \lb \tilde{m}_\epsilon - \tilde{m}, \tilde{m}_\epsilon + \tilde{m} + 2e_3 \rb (\tilde{m}_\epsilon +e_3) + |\tilde{m}+e_3|^2 (\tilde{m}_\epsilon-\tilde{m}) \right|_{\L^2_\rho}^2 \ra \dt \right] \\
		&\lesssim 
		\tilde{\E} \left[ \int_0^T \la 1+|\tilde{m}_\epsilon|^4_{\L^\infty} + |\tilde{m}|^4_{\L^\infty} \ra |\tilde{m}_\epsilon-\tilde{m}|_{\L^2_\rho}^2 \dt \right] 
		\to 0, \quad \text{as } \epsilon \to 0,
	\end{align*}
	which implies $ d_\epsilon((1-|\tilde{m}_\epsilon+e_3|^2)(\tilde{m}_\epsilon+e_3) \rho) \to 0 $ in $ L^{2p}(\tilde{\Omega}; L^2(0,T;\L^2)) $. 
	Thus for $ p_\epsilon $ given in \eqref{eq: defn peps, qeps}, 
	\begin{align*}
		\tilde{\E} \left[ \int_0^T |p_\epsilon(\tilde{m}_\epsilon)|_{\L^2}^2 \dt \right]
		&\lesssim \tilde{\E} \left[ \int_0^T \left|d_\epsilon((1-|\tilde{m}_\epsilon+e_3|^2)(\tilde{m}_\epsilon+e_3) \rho)\right|_{\L^2}^2 \dt \right] \\
		&\quad + \tilde{\E}\left[ \sup_{t \in [0,T]} \la 1+ |\tilde{m}_\epsilon(t)|_{\L^\infty}^4 \ra \int_0^T |d_\epsilon(\tilde{m}_\epsilon) \rho|_{\L^2}^2 \dt \right],
	\end{align*}
	where the right-hand side converges to $ 0 $ as $ \epsilon \to 0 $. 
	Then
	\begin{align*}
		\lim_{\epsilon \to 0} \tilde{\E} \left[ \int_0^T A_\epsilon(\tilde{m}_\epsilon) \dt \right]
		&= \lim_{\epsilon \to 0} \tilde{\E} \left[ \int_0^T \la |d_\epsilon(\tilde{m}_\epsilon)|_{\H^1_\rho}^2 + \sum_k |d_\epsilon(\nabla_{g_k} J_\epsilon \tilde{m}_\epsilon) |_{\L^2_\rho}^2 + |p_\epsilon(\tilde{m}_\epsilon)|_{\L^2}^2 \ra \dt \right] 
		= 0. 
	\end{align*}
	
	Therefore, the right-hand side of \eqref{eq: |tilde m| vs |meps|} converges to $ 0 $ as $ \epsilon \to 0 $, yielding
	\begin{align*}
		\E \left[ \int_0^T \left| 1-|\tilde{m}(t)+e_3|^2 \right|_{\L^2_\rho}^2 \dt \right] \leq 0,
	\end{align*}
	where $ \rho>0 $. 
	In other words, 
	\begin{equation}\label{eq: |tilde m + e3|=1}
		|\tilde{m}(t,x)+e_3|=1, \quad \text{a.e.-}(t,x) \in [0,T] \times \R^d, \ \tilde{\P}\text{-a.s.}
	\end{equation}

\subsection{Convergence of cut-off function}\label{Section: convergence of cut-off}	 	
	Recall \eqref{eq: grad psi}. 
	We define $ \tilde{\psi}^{R,(i)}_\epsilon := \psi_R^{(i)}(|J_\epsilon \tilde{m}_\epsilon+e_3|^2) $ and $ \tilde{\psi}^{R,(i)} := \psi_R^{(i)}(|\tilde{m}+e_3|^2) $, for any non-negative integer $ i $. 
	We simply write \smash{$ \tilde{\psi}^R_\epsilon $} and \smash{$ \tilde{\psi}^R $} when $ i=0 $. 
	In particular, with $ R > 1 $, by \eqref{eq: |tilde m + e3|=1}, 
	\begin{equation}\label{eq: psi^i = 0 for i>=1}
		\tilde{\psi}^{R,(i)}(t,x) 
		= \psi_R^{(i)}(|\tilde{m}(t,x)+e_3|) 
		= \psi_R^{(i)}(1) = 0, \quad \forall i \geq 1,
	\end{equation}
	and thus $ \nabla \tilde{\psi}^R(t,x) = \Delta \tilde{\psi}^R(t,x) = 0 $ for a.e.-$ (t,x) \in [0,T] \times \R^d $, $ \tilde{\P} $-a.s. 
	
	Now we derive some convergence results for $ \tilde{\psi}^R_\epsilon $. 
	By the mean value theorem, we have
	\begin{align*}
		\left|(\tilde{\psi}^{R,(i)}_\epsilon - \tilde{\psi}^{R,(i)}) \rho^\frac{1}{2} \right|_{\L^r}
		&\lesssim \left||J_\epsilon \tilde{m}_\epsilon + e_3|^2 \rho^\frac{1}{2} - |\tilde{m}+e_3|^2 \rho^\frac{1}{2} \right|_{\L^r} \\
		&\lesssim \left|(J_\epsilon \tilde{m}_\epsilon -\tilde{m})\rho^\frac{1}{2} \right|_{\L^2}^{1-\theta} \left|J_\epsilon \tilde{m}_\epsilon - m \right|_{\H^2}^\theta |J_\epsilon \tilde{m}_\epsilon + \tilde{m} + 2 e_3|_{\L^\infty} \\
		&\lesssim |J_\epsilon \tilde{m}_\epsilon -\tilde{m}|_{\L^2_\rho}^{1-\theta} \la |\tilde{m}_\epsilon|_{\H^2}^{1+\theta} + |\tilde{m}|_{\H^2}^{1+\theta} + 1 \ra,
	\end{align*}
	for any non-negative integer $ i $, where $ \theta = \frac{d}{4}-\frac{d}{2r} $ for $ r \in [2,\infty] $. 
	Thus, by \eqref{eq: meps to m}, 
	\begin{equation}\label{eq: psi^i conv Lp, Linf}
		\tilde{\psi}^{R,(i)}_\epsilon \rho^\frac{1}{2} \to \tilde{\psi}^{R,(i)} \rho^\frac{1}{2} \ \text{ in }  L^{2p}(\tilde{\Omega};L^q(0,T;\L^r)), 
	\end{equation}
	for $ p \in [1,\infty) $, $ q \leq \frac{2}{1-\theta} $ and $ r \in [2,\infty] $ as in \eqref{eq: Jmeps to m in Lp, Linf}. 
	In particular, $ \tilde{\psi}^R_\epsilon \to \tilde{\psi}^R $ in $ L^{2p}(\tilde{\Omega}; L^2(0,T;\L^2_\rho)) $. 
	For the gradient of $ \tilde{\psi}^R_\epsilon $, by \eqref{eq: grad psi} and \eqref{eq: psi^i = 0 for i>=1}, 
	\begin{align*}
		\left|\nabla \la \tilde{\psi}^R_\epsilon - \tilde{\psi}^R \ra \right|_{\L^2_\rho} 
		&\lesssim 
		\left|(\tilde{\psi}^{R,(1)}_\epsilon - \tilde{\psi}^{R,(1)}) \rho^\frac{1}{2} \right|_{\L^\infty} |\lb J_\epsilon \tilde{m}_\epsilon+e_3, \nabla J_\epsilon \tilde{m}_\epsilon \rb|_{\L^2} \\
		&\quad + |\tilde{\psi}^{R,(1)}|_{\L^\infty} \left| \lb J_\epsilon \tilde{m}_\epsilon+e_3, \nabla J_\epsilon \tilde{m}_\epsilon \rb - \lb \tilde{m}+e_3, \nabla \tilde{m} \rb \right|_{\L^2_\rho} \\
		&\leq 
		\left|(\tilde{\psi}^{R,(1)}_\epsilon - \tilde{\psi}^{R,(1)}) \rho^\frac{1}{2} \right|_{\L^\infty} \la |\tilde{m}_\epsilon|_{\L^\infty} +1 \ra |\nabla \tilde{m}_\epsilon|_{\L^2} \\
		&\lesssim 
		\left|(\tilde{\psi}^{R,(1)}_\epsilon - \tilde{\psi}^{R,(1)}) \rho^\frac{1}{2} \right|_{\L^\infty} \la |\nabla \tilde{m}_\epsilon|_{\H^1}^2 + |\nabla \tilde{m}_\epsilon|_{\L^2} \ra, 
	\end{align*}
	where the right-hand side converges to $ 0 $ in $ L^{2p}(\tilde{\Omega};L^2(0,T)) $ since $ \tilde{m}_\epsilon \in L^p(\tilde{\Omega};L^\infty(0,T;\H^2)) $ for all $ \epsilon>0 $ and \eqref{eq: psi^i conv Lp, Linf} holds for $ i=1 $. 
	Similarly, for the Laplacian of $ \tilde{\psi}^R_\epsilon $, 
	\begin{align*}
		|\Delta (\tilde{\psi}^R_\epsilon - \tilde{\psi}^R)|_{\L^2_\rho}
		&\lesssim 
		\left|(\tilde{\psi}^{R,(1)}_\epsilon - \tilde{\psi}^{R,(1)}) \rho^\frac{1}{2} \right|_{\L^\infty} \left| |\nabla J_\epsilon \tilde{m}_\epsilon|^2 + \lb J_\epsilon \tilde{m}_\epsilon + e_3, \Delta J_\epsilon \tilde{m}_\epsilon \rb \right|_{\L^2} \\
		&\quad + |\tilde{\psi}^{R,(1)}|_{\L^\infty} \left| |\nabla J_\epsilon \tilde{m}_\epsilon|^2 - |\nabla \tilde{m}|^2 + \lb J_\epsilon \tilde{m}_\epsilon + e_3, \Delta J_\epsilon \tilde{m}_\epsilon \rb - \lb \tilde{m}+e_3, \Delta \tilde{m} \rb \right|_{\L^2_\rho} \\
		&\quad + \left|(\tilde{\psi}^{R,(2)}_\epsilon - \tilde{\psi}^{R,(2)}) \rho^\frac{1}{2} \right|_{\L^\infty} \left| \lb J_\epsilon \tilde{m}_\epsilon+e_3, \nabla J_\epsilon \tilde{m}_\epsilon \rb^2 \right|_{\L^2} \\
		&\quad + |\tilde{\psi}^{R,(2)}|_{\L^\infty} \left| \lb J_\epsilon \tilde{m}_\epsilon+e_3, \nabla J_\epsilon \tilde{m}_\epsilon \rb^2 - \lb \tilde{m}+e_3, \nabla \tilde{m} \rb^2 \right|_{\L^2_\rho} \\
		&\lesssim 
		\left|(\tilde{\psi}^{R,(1)}_\epsilon - \tilde{\psi}^{R,(1)}) \rho^\frac{1}{2} \right|_{\L^\infty} \la |\nabla \tilde{m}_\epsilon|^2_{\H^1} + (|\tilde{m}_\epsilon|_{\L^\infty}+1) |\Delta \tilde{m}_\epsilon|_{\L^2} \ra \\
		&\quad + \left|(\tilde{\psi}^{R,(2)}_\epsilon - \tilde{\psi}^{R,(2)}) \rho^\frac{1}{2} \right|_{\L^\infty} \la |\tilde{m}_\epsilon|_{\L^\infty}+1 \ra^2 |\nabla \tilde{m}_\epsilon|_{\H^1}^2,
	\end{align*} 
	where the right-hand side converges to $ 0 $ in $ L^{2p}(\tilde{\Omega};L^2(0,T)) $. 
	Hence, despite the fact that we only have weak convergence for $ \Delta \tilde{m}_\epsilon $ in $ \L^2_\rho $, we obtain strong convergence for the cut-off function in $ \H^2 $:
	\begin{equation}\label{eq: psi-eps conv H2loc}
		\tilde{\psi}^R_\epsilon \rho^\frac{1}{2} \to \tilde{\psi}^R \rho^\frac{1}{2} \ \text{ in } L^{2p}(\tilde{\Omega};L^2(0,T;\H^2)),
	\end{equation}

\subsection{Convergence of drift and diffusion coefficients}\label{Section: convergence of drift and diffusion coefficients}
	Using results in previous sections, we prove in Lemma \ref{Lemma: conv Feps,Seps,Geps} weak convergences for the drift coefficients $ F^R_\epsilon $ and $ S_{k,\epsilon} $, and strong convergence for the diffusion coefficient $ G_{k,\epsilon} $. 
	\begin{lemma}\label{Lemma: conv Feps,Seps,Geps}
		For $ p \in [1,\infty) $ and $ \varphi \in L^{2p}(\tilde{\Omega}; L^4(0,T; \mathcal{C}^\infty_c(\R^d))) $, 
		\begin{align*}
			\lim_{\epsilon \to 0} \tilde{\E} \left[ \la \int_0^T \lb F^R_\epsilon(\tilde{m}_\epsilon) - F(\tilde{m}+e_3), \varphi \rb_{\L^2_\rho}(t) \dt \ra^p \right] &= 0, \\
			\lim_{\epsilon \to 0} \tilde{\E} \left[ \la \int_0^T \sum_k \lb S_{k,\epsilon}(\tilde{m}_\epsilon) - S_k(\tilde{m}), \varphi \rb_{\L^2_\rho}(t) \dt \ra^p \right] &= 0,
		\end{align*}
		and
		\begin{align*}
			\lim_{\epsilon \to 0} \tilde{\E} \left[ \la \int_0^T \sum_k |G_{k,\epsilon}(\tilde{m}_\epsilon) - G_k(\tilde{m})|_{\L^2_\rho}^2(t) \dt \ra^p \right] = 0. 
		\end{align*}
	\end{lemma}
	\begin{proof}
		Let $ \tilde{u}_\epsilon := J_\epsilon \tilde{m}_\epsilon $. 
		We have
		\begin{align*}
			\lb F^R_\epsilon(\tilde{m}_\epsilon) - F(\tilde{m}+e_3), \varphi \rb_{\L^2_\rho} 
			&= \lb J_\epsilon \la \tilde{\psi}_\epsilon^R \bar{F}(\tilde{u}_\epsilon+e_3) \ra - \bar{F}(\tilde{m}+e_3) - (J_\epsilon \nabla_v \tilde{u}_\epsilon - \nabla_v \tilde{m}), \varphi \rho \rb_{\L^2}
		\end{align*}
		where $ \tilde{\psi}^R = 1 $, a.e.-$ (t,x) \in [0,T] \times \R^d $, $ \tilde{\P} $-a.s. and
		\begin{align*}
			\lb J_\epsilon \la \tilde{\psi}_\epsilon^R \bar{F}(\tilde{u}_\epsilon+e_3) \ra - \bar{F}(\tilde{m}+e_3), \varphi \rho \rb_{\L^2} 
			&= 
			\lb \bar{F}(\tilde{u}_\epsilon+e_3), \la \tilde{\psi}_\epsilon^R - \tilde{\psi}^R \ra \varphi \rho \rb_{\L^2} \\
			&\quad + \lb \bar{F}(\tilde{u}_\epsilon+e_3) - \bar{F}(\tilde{m}+e_3), \varphi \rho \rb_{\L^2} \\
			&\quad + \lb \tilde{\psi}_\epsilon^R \bar{F}(\tilde{u}_\epsilon+e_3) , d_\epsilon(\varphi \rho) \rb_{\L^2} \\
			&=: 
			{\rm I}_{0,\bar{F}} + {\rm I}_{1,\bar{F}} + {\rm I}_{2,\bar{F}}.
		\end{align*} 
		Similarly, we define $ \{ {\rm I}_{1,\nabla_v}, {\rm I}_{2,\nabla_v}\}, \{ {\rm I}_{1,S_k}, {\rm I}_{2,S_k} \} $ with $ \nabla_v, S_k $ in place of $ \bar{F} $ and $ 1 $ in place of $ \tilde{\psi}_\epsilon^R $, for $ k \geq 1 $.

		{\bf Convergence of $ F_\epsilon^R $.} 
		\begin{enumerate}[label=(\roman*), wide, labelwidth=!, labelindent=0pt]
			\item 
			For $ {\rm I}_{0,\bar{F}} $, 
			let $ f_\epsilon = \la \tilde{\psi}_\epsilon^R - \tilde{\psi}^R \ra \varphi \rho $.   
			Thus, 
			\begin{align*}
				\lb \bar{F}(\tilde{u}_\epsilon+e_3), f_\epsilon \rb_{\L^2} 
				&= 
				\lb \Delta \tilde{u}_\epsilon + \alpha (\tilde{u}_\epsilon+e_3) \times \Delta \tilde{u}_\epsilon, f_\epsilon \times (\tilde{u}_\epsilon+e_3) \rb_{\L^2} \\
				&\quad - h \lb e_3 + \alpha \tilde{u}_\epsilon \times e_3,  f_\epsilon \times (\tilde{u}_\epsilon+e_3) \rb_{\L^2} \\
				&\quad + \lb (\tilde{u}_\epsilon+e_3) \times \Delta \tilde{u}_\epsilon, \Delta f_\epsilon \rb_{\L^2} + 2 \lb \Delta \tilde{u}_\epsilon, \nabla f_\epsilon \times \nabla \tilde{u}_\epsilon \rb_{\L^2} \\
				&\quad + \alpha \lb (\tilde{u}_\epsilon+e_3) \times \Delta \tilde{u}_\epsilon, f_\epsilon \times \Delta \tilde{u}_\epsilon + \Delta f_\epsilon \times (\tilde{u}_\epsilon+e_3) + 2 \nabla f_\epsilon \times \nabla \tilde{u}_\epsilon \rb_{\L^2} \\
				&\quad + 2 \alpha \lb \nabla \tilde{u}_\epsilon \times \Delta \tilde{u}_\epsilon, \nabla f_\epsilon \times (\tilde{u}_\epsilon+e_3) + f_\epsilon \times \nabla \tilde{u}_\epsilon \rb_{\L^2} \\
				&\quad - \gamma \lb (\tilde{u}_\epsilon+e_3) \times \nabla_v u, f_\epsilon \rb_{\L^2} \\
				&\leq 
				\la |\tilde{u}_\epsilon+e_3|_{\L^\infty} + \alpha |\tilde{u}_\epsilon+e_3|_{\L^\infty}^2 \ra |\Delta \tilde{u}_\epsilon|_{\L^2} |f_\epsilon|_{\L^2} \\
				&\quad + h \la |\tilde{u}_\epsilon|_{\L^2} + \alpha |\tilde{u}_\epsilon|_{\L^4}^2 \ra |f_\epsilon|_{\L^2} \\
				&\quad + |\tilde{u}_\epsilon+e_3|_{\L^\infty} |\Delta m_\epsilon|_{\L^2} |\Delta f_\epsilon|_{\L^2} 
				+ 2 |\Delta \tilde{u}_\epsilon|_{\L^2} |\nabla \tilde{u}_\epsilon|_{\L^4} |\nabla f_\epsilon|_{\L^4} \\
				&\quad + \alpha |\tilde{u}_\epsilon+e_3|_{\L^\infty} |\Delta \tilde{u}_\epsilon|_{\L^2}^2 |f_\epsilon|_{\L^\infty} 
				+ \alpha |\tilde{u}_\epsilon+e_3|_{\L^\infty}^2 |\Delta \tilde{u}_\epsilon|_{\L^2} |\Delta f_\epsilon|_{\L^2} \\
				&\quad + 4 \alpha |\tilde{u}_\epsilon+e_3|_{\L^\infty} |\Delta \tilde{u}_\epsilon|_{\L^2} |\nabla \tilde{u}_\epsilon|_{\L^4} |\nabla f_\epsilon|_{\L^4} 
				+ 2 \alpha |\Delta \tilde{u}_\epsilon|_{\L^2} |\nabla \tilde{u}_\epsilon|_{\L^4}^2 |f_\epsilon|_{\L^\infty} \\
				&\quad + |\gamma v|_{L^\infty(\R^d;\R^d)} |\tilde{u}_\epsilon+e_3|_{\L^\infty} |\nabla \tilde{u}_\epsilon|_{\L^2} |f_\epsilon|_{\L^2} \\
				&\lesssim 
				\la 1+|\tilde{m}_\epsilon|_{\H^2} + |\tilde{m}_\epsilon|_{\H^2}^2 + |\tilde{m}_\epsilon|_{\H^2}^3 \ra |f_\epsilon|_{\H^2} \\
				&\lesssim 
				(1+|\tilde{m}_\epsilon|_{\H^2}^3) |f_\epsilon|_{\H^2}.
			\end{align*}
			Then by \eqref{eq: psi^i conv Lp, Linf} with $ i=0 $, \eqref{eq: psi-eps conv H2loc}, and that $ \varphi \in L^{2p}(\tilde{\Omega};L^2(0,T;\H^2)) $, 
			\begin{align*}
				|f_\epsilon|_{\H^2}
				&= \left|\la \tilde{\psi}_\epsilon^R - \tilde{\psi}^R \ra \varphi \rho \right|_{\H^2} \\
				&\lesssim \left|(\tilde{\psi}_\epsilon^R - \tilde{\psi}^R) \rho \right|_{\H^2} \left|\varphi \right|_{\L^\infty} 
				+ \left|(\tilde{\psi}_\epsilon^R - \tilde{\psi}^R) \rho \right|_{\L^\infty} \left|\nabla \varphi \right|_{\H^1} 
				+ \left|\nabla \la (\tilde{\psi}_\epsilon^R - \tilde{\psi}^R) \rho \ra \right|_{\L^4} \left|\nabla \varphi \right|_{\L^4} \\
				&\to 0 \quad \text{ in } L^{2p}(\tilde{\Omega};L^1(0,T)). 
			\end{align*}
			This implies
			\begin{align*}
				\tilde{\E} \left[ \left| \int_0^T {\rm I}_{0,\bar{F}} \dt \right|^p \right] 
				&\lesssim \tilde{\E} \left[ \la \int_0^T (1+|\tilde{m}_\epsilon|_{\H^2}^3) |f_\epsilon|_{\H^2} \dt \ra^p \right] \\
				&\leq
				\tilde{\E} \left[ \sup_{t \in [0,T]} \la 1+|\tilde{m}_\epsilon|_{\H^2}^{6p} \ra \right]^\frac{1}{2} 
				\tilde{\E} \left[ \la \int_0^T |f_\epsilon|_{\H^2} \dt \ra^{2p} \right]^\frac{1}{2} \to 0, 
			\end{align*}
			where the first expectation in the last inequality is finite for all $ \epsilon>0 $ since $ \tilde{m}_\epsilon $ has the same $ L^\infty(0,T;\H^2) $-estimate as $ m_\epsilon $. 
			
			\item 
			For $ {\rm I}_{1,\bar{F}} $, we focus on the difference $ \bar{F}(\tilde{u}_\epsilon+e_3) - \bar{F}(\tilde{m}+e_3) $:
			\begin{align*}
				&\bar{F}(\tilde{u}_\epsilon+e_3) - \bar{F}(\tilde{m}+e_3) \\
				&= - \left[ (\tilde{u}_\epsilon - \tilde{m}) \times \la he_3 + \alpha \tilde{u}_\epsilon \times he_3 \ra 
				+ \alpha (\tilde{m}+e_3) \times \la (\tilde{u}_\epsilon - \tilde{m}) \times he_3 \ra \right] \\
				&\quad - \gamma \left[ (\tilde{u}_\epsilon - \tilde{m}) \times \nabla_v u_\epsilon + (\tilde{m}+e_3) \times \nabla_v(\tilde{u}_\epsilon - m) \right] \\
				&\quad + \left[ (\tilde{u}_\epsilon - \tilde{m}) \times \la \Delta \tilde{u}_\epsilon + \alpha (\tilde{u}_\epsilon+e_3) \times \Delta \tilde{u}_\epsilon \ra 
				+ \alpha (\tilde{m}+e_3) \times \la (\tilde{u}_\epsilon - \tilde{m}) \times \Delta \tilde{u}_\epsilon \ra \right] \\
				&\quad + (\tilde{m}+e_3) \times \la \Delta(\tilde{u}_\epsilon - \tilde{m}) + \alpha (\tilde{m}+e_3) \times \Delta(\tilde{u}_\epsilon - \tilde{m}) \ra \\
				&\quad + (\tilde{u}_\epsilon - \tilde{m}) \times \la \Delta^2 \tilde{u}_\epsilon + \alpha (\tilde{u}_\epsilon+e_3) \times \Delta^2 \tilde{u}_\epsilon \ra \\
				&\quad + \alpha (\tilde{m}+e_3) \times \la (\tilde{u}_\epsilon - \tilde{m}) \times \Delta^2 \tilde{u}_\epsilon \ra \\
				&\quad + (\tilde{m}+e_3) \times \la \Delta^2(\tilde{u}_\epsilon - \tilde{m}) + \alpha (\tilde{m}+e_3) \times \Delta^2(\tilde{u}_\epsilon - \tilde{m}) \ra \\
				&=: 
				\sum_{i=1}^7 {\rm H}_i.
			\end{align*}
			For $ i=1,2,3 $, we deduce $ \int_0^T \langle {\rm H}_i, \varphi \rho \rangle_{\L^2} \dt \to 0 $ in $ L^{p}(\tilde{\Omega}) $ from \eqref{eq: meps to m}, \eqref{eq: |tilde m + e3|=1} and the moment estimates of $ \tilde{m}_\epsilon, \tilde{m} $ in $ L^{2p}(\tilde{\Omega}; L^\infty(0,T; \H^2)) $. 
			
			For $ i=4 $, similarly $ \int_0^T \langle {\rm H}_4, \varphi \rho \rangle_{\L^2} \dt \to 0 $ in $ L^{p}(\tilde{\Omega}) $ since
			$ \Delta \tilde{m}_\epsilon \rightharpoonup \Delta \tilde{m} $ in $ L^{2p}(\tilde{\Omega}; L^2(0,T;\L^2_\rho)) $ 
			and thanks to \eqref{eq: |tilde m + e3|=1}, 
			\begin{align*}
				\varphi \times (\tilde{m}+e_3) + \alpha \la \varphi \times (\tilde{m}+e_3) \ra \times (\tilde{m}+e_3) \in L^2(\tilde{\Omega}; L^2(0,T;\L^2)),
			\end{align*}

			Now we estimate $ {\rm H}_5 $, $ {\rm H}_6 $ and $ {\rm H}_7 $.  
			\begin{align*}
				\lb {\rm H}_5, \varphi \rb_{\L^2_\rho}
				&= \lb (\tilde{u}_\epsilon - \tilde{m}) \times \Delta^2 \tilde{u}_\epsilon, \varphi \rb_{\L^2_\rho} \\
				&\quad + \alpha \lb \tilde{u}_\epsilon - \tilde{m}, \tilde{\psi}^R_\epsilon \la (\tilde{u}_\epsilon+e_3) \times \Delta^2 \tilde{u}_\epsilon \ra \times \varphi \rb_{\L^2_\rho} \\
				&\quad + \alpha \lb \tilde{u}_\epsilon - \tilde{m}, (\tilde{\psi}^R-\tilde{\psi}^R_\epsilon) \la (\tilde{u}_\epsilon+e_3) \times \Delta^2 \tilde{u}_\epsilon \ra \times \varphi \rb_{\L^2_\rho} \\
				&=: 
				{\rm H}_{5a} + \alpha {\rm H}_{5b} - \alpha {\rm H}_{5c}.
			\end{align*}
			For $ {\rm H}_{5a} $, integrating-by-parts, 
			\begin{align*}	
				{\rm H}_{5a}
				&= \lb \Delta (\tilde{u}_\epsilon - \tilde{m}), \Delta \tilde{m} \times \varphi \rho \rb_{\L^2} 
				+ \lb \tilde{u}_\epsilon - \tilde{m}, \Delta \tilde{u}_\epsilon \times \Delta (\varphi \rho) \rb_{\L^2} 
				+ 2\lb \nabla \tilde{u}_\epsilon - \tilde{m}, \Delta \tilde{u}_\epsilon \times \nabla (\varphi \rho) \rb_{\L^2} \\
				&\leq 
				\lb \Delta (\tilde{u}_\epsilon - \tilde{m}), \Delta \tilde{m} \times \varphi \rb_{\L^2_\rho}
				+ c |(\tilde{u}_\epsilon - \tilde{m}) \rho|_{\L^\infty} |\Delta \tilde{u}_\epsilon|_{\L^2} |\Delta (\varphi \rho) \rho^{-1}|_{\L^2} \\
				&\quad + c |\nabla (\tilde{u}_\epsilon - \tilde{m}) \rho|_{\L^4} |\Delta \tilde{u}_\epsilon|_{\L^2} |\nabla (\varphi \rho) \rho^{-1}|_{\L^4}. 
			\end{align*}
			where $ \Delta \tilde{m} \times \varphi \in L^{2p}(\tilde{\Omega}; L^2(0,T;\L^2)) $ and 
			\begin{equation}\label{eq: (u-m)rho L4}
				|\nabla (\tilde{u}_\epsilon - \tilde{m}) \rho|_{\L^4} 
				\lesssim 
				|\nabla(\tilde{u}_\epsilon - \tilde{m}) |_{\L^2_\rho}^{1-\frac{d}{4}} 
				\la |\Delta \tilde{m}_\epsilon|_{\L^2}^\frac{d}{4} + |\Delta \tilde{m}|_{\L^2}^\frac{d}{4} \ra 
				\to 0 \quad \text{ in } L^{2p}(\tilde{\Omega};L^4(0,T)). 
			\end{equation}
			Then, the convergence $ \int_0^T {\rm H}_{5a} \dt \to 0 $ in $ L^{p}(\tilde{\Omega}) $ follows from \eqref{eq: rho^-1 D-D2}, \eqref{eq: meps to m} and \eqref{eq: Jmeps to m in Lp, Linf}.  	
			Moreover,
			\begin{align*}
				{\rm H}_{5b}
				&\leq |(\tilde{\psi}^R_\epsilon)^\frac{1}{2} (\tilde{u}_\epsilon+e_3) \times \Delta^2 \tilde{u}_\epsilon |_{\L^2} |\tilde{u}_\epsilon - \tilde{m}|_{\L^2_\rho} | (\tilde{\psi}^R_\epsilon)^\frac{1}{2} \varphi|_{\L^\infty},
			\end{align*}
			where by Lemma \ref{Lemma: meps H2-est} and \eqref{eq: meps to m}, $ \int_0^T {\rm H}_{5b} \dt \to 0 $ in $ L^{p}(\tilde{\Omega}) $. 	
			For $ {\rm H}_{5c} $, recall $ f_\epsilon = (\psi^R_\epsilon-\psi^R) \varphi \rho $ from part (i), where $ f_\epsilon \to 0 \in L^{2p}(\tilde{\Omega};L^1(0,T;\H^2)) $. 
			Then, 
			\begin{align*}
				{\rm H}_{5c}
				&= \lb \tilde{u}_\epsilon - \tilde{m}, \la (\tilde{u}_\epsilon+e_3) \times \Delta^2 \tilde{u}_\epsilon \ra \times f_\epsilon \rb_{\L^2} \\
				&= 
				\lb (\tilde{u}_\epsilon+e_3) \times \Delta \tilde{u}_\epsilon, \Delta f_\epsilon \times (\tilde{u}_\epsilon - \tilde{m}) + f_\epsilon \times \Delta (\tilde{u}_\epsilon - \tilde{m}) + 2 \nabla f_\epsilon \times \nabla(\tilde{u}_\epsilon - \tilde{m}) \rb_{\L^2} \\
				&\quad + 2 \lb \nabla \tilde{u}_\epsilon \times \Delta \tilde{u}_\epsilon, \nabla f_\epsilon \times (\tilde{u}_\epsilon - \tilde{m}) + f_\epsilon \times \nabla (\tilde{u}_\epsilon - \tilde{m}) \rb_{\L^2} \\
				&\leq 
				(1+|\tilde{u}_\epsilon|_{\L^\infty}) |\Delta \tilde{u}_\epsilon|_{\L^2} \la |\tilde{u}_\epsilon -\tilde{m}|_{\L^\infty} |\Delta f_\epsilon|_{\L^2} 
				+ |\Delta(\tilde{u}_\epsilon - \tilde{m})|_{\L^2} |f_\epsilon|_{\L^\infty} 
				+ 2|\nabla(\tilde{u}_\epsilon - \tilde{m})|_{\L^4} |\nabla f_\epsilon|_{\L^4} \ra \\
				&\quad + 2 |\nabla \tilde{u}_\epsilon|_{\L^4} |\Delta \tilde{u}_\epsilon|_{\L^2} \la |\nabla f_\epsilon|_{\L^4} |\tilde{u}_\epsilon - \tilde{m}|_{\L^\infty} + |f_\epsilon|_{\L^\infty} |\nabla (\tilde{u}_\epsilon - \tilde{m})|_{\L^4} \ra.
			\end{align*}
			Since $ \tilde{u}_\epsilon, \tilde{m} \in L^{2p}(\tilde{\Omega}; L^\infty(0,T;\H^2)) $, 
			$ \int_0^T {\rm H}_{5c} \dt \to 0 $ in $ L^{p}(\tilde{\Omega}) $ by the convergence of $ f_\epsilon $ as in part (i). 	
			
			For $ {\rm H}_6 $, recall \eqref{eq: |tilde m + e3|=1}. Then,
			\begin{align*}
				\alpha^{-1} \lb {\rm H}_6, \varphi \rho \rb_{\L^2}
				&= \lb \Delta^2 \tilde{u}_\epsilon, \la \varphi \rho \times (\tilde{m}+e_3) \ra \times (\tilde{u}_\epsilon - \tilde{m}) \rb_{\L^2} \\
				&= 
				\lb (\tilde{u}_\epsilon - \tilde{m}), \Delta \tilde{u}_\epsilon \times \la \Delta (\varphi \rho) \times (\tilde{m}+e_3) + 2 \nabla(\varphi \rho) \times \nabla \tilde{m} \ra \rb_{\L^2} \\
				&\quad + \lb (\tilde{u}_\epsilon - \tilde{m}), \Delta \tilde{u}_\epsilon \times \la \varphi \times \Delta \tilde{m} \ra \rb_{\L^2_\rho} \\
				&\quad + \lb \Delta (\tilde{u}_\epsilon-\tilde{m}), \Delta \tilde{m} \times \la \varphi \times (\tilde{m}+e_3) \ra \rb_{\L^2_\rho} \\
				&\quad + 2\lb \nabla (\tilde{u}_\epsilon - \tilde{m}), \Delta \tilde{u}_\epsilon \times \la \nabla(\varphi \rho) \times (\tilde{m}+e_3) + \varphi \rho \times \nabla \tilde{m} \ra \rb_{\L^2} \\ 
				&\lesssim 
				|(\tilde{u}_\epsilon - \tilde{m}) \rho|_{\L^\infty} |\Delta \tilde{u}_\epsilon|_{\L^2} \la |\Delta (\varphi \rho) \rho^{-1}|_{\L^2} + |\nabla \tilde{m}|_{\L^4} |\nabla (\varphi \rho) \rho^{-1}|_{\L^4} \ra \\\
				&\quad + |(\tilde{u}_\epsilon - \tilde{m}) \rho|_{\L^\infty} |\Delta \tilde{u}_\epsilon|_{\L^2} |\Delta \tilde{m}|_{\L^2} |\varphi|_{\H^2} \\
				&\quad + \lb \Delta (\tilde{u}_\epsilon-\tilde{m}), \Delta \tilde{m} \times \la \varphi \times (\tilde{m}+e_3) \ra \rb_{\L^2_\rho} \\
				&\quad + |\nabla (\tilde{u}_\epsilon - \tilde{m}) \rho|_{\L^4} |\Delta \tilde{u}_\epsilon|_{\L^2} \la |\nabla (\varphi \rho) \rho^{-1}|_{\L^4} + |\nabla \tilde{m}|_{\L^4} |\varphi|_{\H^2} \ra,
			\end{align*}
			where the third term on the right-hand side converges by the weak convergence in \eqref{eq: meps to m} and the fact $ \Delta \tilde{m} \times (\varphi \times (\tilde{m}+e_3)) \in L^{2p}(\tilde{\Omega}; L^2(0,T;\L^2)) $. 
			Thus, $ \int_0^T \lb {\rm H}_6, \varphi \rho \rb_{\L^2} \dt \to 0 $ in $ L^{p}(\tilde{\Omega}) $ holds by \eqref{eq: rho^-1 D-D2}, \eqref{eq: Jmeps to m in Lp, Linf} and \eqref{eq: (u-m)rho L4}.

			For $ {\rm H}_7 $, we have
			\begin{align*}
				\lb {\rm H}_7, \varphi \rho \rb_{\L^2} 
				&= \lb \Delta^2(\tilde{u}_\epsilon - \tilde{m}), \varphi \rho \times (\tilde{m}+e_3) + \alpha (\tilde{m}+e_3) \times ((\tilde{m}+e_3) \times \varphi \rho) \rb_{\L^2} \\
				&= \lb \Delta(\tilde{u}_\epsilon - \tilde{m}), \Delta \la \varphi \rho \times (\tilde{m}+e_3) + \alpha (\tilde{m}+e_3) \times ((\tilde{m}+e_3) \times \varphi \rho) \ra \rho^{-1} \rb_{\L^2_\rho},
			\end{align*}
			where $ \Delta (\varphi \rho \times (\tilde{m}+e_3) + \alpha (\tilde{m}+e_3) \times ((\tilde{m}+e_3) \times \varphi \rho)) \rho^{-1} \in L^2(\tilde{\Omega}; L^2(0,T;\L^2)) $. 
			Then $ \int_0^T \langle {\rm H}_7, \varphi \rho \rangle_{\L^2} \dt \to 0 $ in $ L^{p}(\tilde{\Omega}) $ follows from the weak convergence in \eqref{eq: meps to m} .  
			
			Therefore, as $ \epsilon \to 0 $, for $ p \in [1,\infty) $,
			\begin{align*}
				\tilde{\E} \left[ \left| \int_0^T {\rm I}_{1,\bar{F}} \dt \right|^p \right] 
				\lesssim \sum_{i=1}^7 \tilde{\E} \left[ \left| \int_0^T \lb {\rm H}_i, \varphi \rho \rb_{\L^2} \dt \right|^p \right] 
				\to 0. 
			\end{align*}

			\item 
			For $ {\rm I}_{2,\bar{F}} $, recall that for any $ p \in [1,\infty) $, 
			\begin{align*}
				\sup_{\epsilon>0} \tilde{\E} \left[ \la \int_0^T \left|(\tilde{\psi}_\epsilon^R)^{\frac{1}{2}}  (\tilde{u}_\epsilon+e_3) \times \Delta^2  \tilde{u}_\epsilon \right|^2_{\L^2}(t) \dt \ra^p \right] < \infty.
			\end{align*}
			Then it is clear from \eqref{eq: m in Linft-H2} that 
			\begin{align*}
				\sup_{\epsilon>0} \tilde{\E} \left[ \la \int_0^T \left|\tilde{\psi}_\epsilon^R \bar{F}(\tilde{u}_\epsilon+e_3) \right|_{\L^2}^2(t) \dt \ra^p \right] < \infty. 
			\end{align*}
			Since $ {\rm I}_{2,\bar{F}} \leq |\tilde{\psi}_\epsilon^R \bar{F}(\tilde{u}_\epsilon+e_3) |_{\L^2} |d_\epsilon(\varphi \rho)|_{\L^2} $, by the approximation property of $ J_\epsilon $, we have $ d_\epsilon(\varphi \rho) \to 0 $ in $ L^{2p}(\tilde{\Omega};L^2(0,T;\L^2)) $ and thus $ \int_0^T {\rm I}_{2,\bar{F}} \dt \to 0 $ in $ L^p(\tilde{\Omega}) $.
			
			Moreover, since $ \tilde{\psi}_\epsilon^R (\tilde{u}_\epsilon+e_3) \times \Delta^2 \tilde{u}_\epsilon $ 
			and $ \tilde{\psi}_\epsilon^R (\tilde{u}_\epsilon+e_3) \times ((\tilde{u}_\epsilon+e_3) \times \Delta^2 \tilde{u}_\epsilon) $ are in 
			$ L^{2p}(\tilde{\Omega};L^2(0,T;\L^2)) $, there exist measurable processes $ Y,Z \in L^{2p}(\tilde{\Omega};L^2(0,T;\L^2)) $ such that 
			\begin{align*}
				\tilde{\psi}_\epsilon^R (\tilde{u}_\epsilon+e_3) \times \Delta^2 \tilde{u}_\epsilon &\rightharpoonup Y \quad \text{ in } L^{2p}(\tilde{\Omega};L^2(0,T;\L^2)), \\
				\tilde{\psi}_\epsilon^R (\tilde{u}_\epsilon +e_3) \times \la (\tilde{u}_\epsilon+e_3) \times \Delta^2 \tilde{u}_\epsilon \ra &\rightharpoonup Z \quad 
				\text{ in } L^{2p}(\tilde{\Omega};L^2(0,T;\L^2)).
			\end{align*}
			By the arguments in previous steps and the uniqueness of weak limit, we have $ Y= (\tilde{m}+e_3) \times \Delta^2 \tilde{m} $ and $ Z = (\tilde{m}+e_3) \times ((\tilde{m}+e_3) \times \Delta^2 \tilde{m}) $ in $ L^{2p}(\tilde{\Omega};L^2(0,T;\L^2)) $ in the weak sense \eqref{eq: cross Delta^2 weak form}.

			\item 
			Similar calculations hold for $ {\rm I}_{1,\nabla_v} $ and $ {\rm I}_{2,\nabla_v} $. 
			By the uniform integrability and strong convergence of $ \tilde{m}_\epsilon $ in $ L^{2p}(\tilde{\Omega}; L^2(0,T;\H^1_\rho)) $ and properties of $ J_\epsilon $, 
			\begin{align*}
				\tilde{\E} \left[ \int_0^T \lb \nabla_v (\tilde{u}_\epsilon - \tilde{m}), \varphi \rb_{\L^2_\rho} \dt \right]
				&\leq \tilde{\E} \left[ \int_0^T |v|_{L^\infty(\R^d;\R^d)} |\nabla (\tilde{u}_\epsilon - \tilde{m})|_{\L^2_\rho} |\varphi|_{\L^2} \dt \right] \to 0, \\
				\tilde{\E} \left[ \int_0^T \lb \nabla_v \tilde{u}_\epsilon, d_\epsilon(\varphi \rho) \rb_{\L^2} \dt \right]
				&\leq 
				|v|_{L^\infty(\R^d; \R^d)} 
				\la \sup_{\epsilon>0} \tilde{\E} \left[ |\nabla \tilde{u}_\epsilon|_{L^\infty(0,T;\L^2)}^2 \right] \ra^\frac{1}{2} 
				\tilde{\E} \left[\int_0^T |d_\epsilon(\varphi \rho)|_{\L^2}^2 \dt \right]^\frac{1}{2} \to 0, 
			\end{align*}
			as $ \epsilon \to 0 $. The same arguments follow for convergence in $ L^p(\tilde{\Omega}) $, $ p \in (1,\infty) $. 
		\end{enumerate}

		{\bf Convergence of $ S_{k,\epsilon} $.} 
		
		For $ {\rm I}_{1,S_k} $,
		\begin{align*}
			{\rm I}_{1,S_k} 
			&= \lb S_k(\tilde{u}_\epsilon) - S_k(\tilde{m}), \varphi \rb_{\L^2_\rho} \\
			&= \lb (\nabla_{g_k})^2(\tilde{u}_\epsilon - \tilde{m}), \varphi \rho \rb_{\L^2} \\
			&= -\lb \nabla_{g_k}(\tilde{u}_\epsilon - \tilde{m}), \nabla_{g_k}(\varphi \rho) + ({\rm div} g_k) \varphi \rho \rb_{\L^2} \\
			&\lesssim |g_k|_{W^{1,\infty}(\R^d;\R^d)} |\nabla (\tilde{u}_\epsilon - \tilde{m}) \rho|_{\L^2} \la |\nabla (\varphi \rho) \rho^{-1}|_{\L^2} + |\varphi|_{\L^2} \ra.
		\end{align*}
		Then $ \sum_k \int_0^T {\rm I}_{1,S_k} \dt \to 0 $ in $ L^p(\tilde{\Omega}) $ by \eqref{eq: rho^-1 D-D2}, \eqref{eq: g in W2inf} and \eqref{eq: meps to m}. 
		
		As in the case of $ {\rm I}_{2,\bar{F}} $, the convergence of $ \sum_k {\rm I}_{2,S_k} = \sum_k \langle S_k(\tilde{u}_\epsilon), d_\epsilon(\varphi \rho) \rangle_{\L^2} $ follows from \eqref{eq: g in W2inf}, the $ L^2(0,T;\H^2) $-integrability of $ \tilde{u}_\epsilon $ and properties of $ J_\epsilon $.

		{\bf Convergence of $ G_{k,\epsilon} $.} 
		
		By \eqref{eq: |Ju|_L2rho <= c|u|_L2rho}, 
		\begin{align*}
			|G_{k,\epsilon}(\tilde{m}_\epsilon) - G_k(\tilde{m})|_{\L^2_\rho}
			&\leq 
			|J_\epsilon (G_k(\tilde{u}_\epsilon) - G_k(\tilde{m}))|_{\L^2_\rho} 
			+ |J_\epsilon G_k(\tilde{m}) - G_k(\tilde{m}))|_{\L^2_\rho} \\
			&\leq 
			c|G_k(\tilde{u}_\epsilon) - G_k(\tilde{m})|_{\L^2_\rho} 
			+ |d_\epsilon (G_k(\tilde{m})) - G_k(\tilde{m}))|_{\L^2}.
		\end{align*}
		Since $ G_k(\tilde{m}) = -\nabla_{g_k} \tilde{m} \in L^{2p}(\tilde{\Omega}; L^\infty(0,T;\L^2)) $, and 
		\begin{align*}
			|G_k(\tilde{u}_\epsilon) - G_k(\tilde{m})|_{\L^2_\rho} &\leq |g_k|_{L^\infty(\R^d;\R^d)} |\nabla (\tilde{u}_\epsilon - \tilde{m})|_{\L^2_\rho},
		\end{align*}
		we obtain convergence for $ \smash{\sum_k |G_{k,\epsilon}(\tilde{m}_\epsilon) - G_k(\tilde{m})|_{\L^2_\rho}^2} $ in $ L^p(\tilde{\Omega};L^1(0,T)) $ from \eqref{eq: g in W2inf}, \eqref{eq: meps to m} and approximation properties of $ J_\epsilon $.  		
	\end{proof}
	
	By the density of $ L^4(0,T;\mathcal{C}_c^\infty(\R^d)) $ in $ L^2(0,T;\L^2_\rho) $, the weak convergences in Lemma \ref{Lemma: conv Feps,Seps,Geps} reduce to 
	\begin{align*}
		F^R_\epsilon(\tilde{m}_\epsilon) + \frac{1}{2} \sum_k S_{k,\epsilon}(\tilde{m}_\epsilon) 
		\rightharpoonup F(\tilde{m}+e_3) + \frac{1}{2} \sum_k S_k(\tilde{m}) \quad 
		\text{in } L^{2p}(\tilde{\Omega};L^2(0,T;\L^2_\rho)).
	\end{align*} 

	\section{Proof of Theorem \ref{Theorem: E!}}\label{Section: Proof of Theorem}
	
	\subsection{Existence of martingale solution}\label{Section: existence}	
	As in \cite[Section 5]{BrzezniakGoldysJegaraj_2013}, we can show that the limit process $ (\tilde{m},\tilde{W}) $ on $ (\tilde{\Omega}, \tilde{\mathcal{F}}, \tilde{\P}) $ from Section \ref{Section: convergence for fixed R} is a martingale solution of the equation
	\begin{equation}\label{eq: sLLG m}
		\tilde{m}(t)
		= \tilde{m}_0 + \int_0^t \bar{F}(\tilde{m}+e_3)(s) \ds + \frac{1}{2} \sum_k \int_0^t S_k(\tilde{m})(s) \ds + \sum_k \int_0^t G_k(\tilde{m})(s) \ \d \tilde{W}_k(s).
	\end{equation}
	We only outline the arguments here. 
	Since $ \{\tilde{W}_\epsilon\} $ have the same laws as $ W $, the processes $ \tilde{W}_\epsilon $ and thus $ \tilde{W} $ are Wiener processes with covariance $ Q $ on $ (\tilde{\Omega}, \tilde{\mathcal{F}}, \tilde{\P}) $. Using the pointwise convergence $ (\tilde{m}_\epsilon, \tilde{W}_\epsilon) \to (\tilde{m}, \tilde{W}) $ in $ E \times \mathcal{C}([0,T];H^4(\R^d;\R^d)) $, Lemma \ref{Lemma: conv Feps,Seps,Geps} and the embedding $ \L^2_\rho \Subset \H^{-1} $, we deduce from the uniqueness of weak limit in $ L^2(\tilde{\Omega}; \H^{-1}) $ that $ \tilde{m}(t) $ satisfies \eqref{eq: sLLG m} (in the weak sense) for every $ t \in [0,T] $ and 
	\begin{align*}
		\E \left[ |\tilde{m}(t)|_{L^\infty(0,T;\H^2)}^2 + \left|(\tilde{m}+e_3) \times \Delta^2 \tilde{m}\right|_{L^2(0,T;\L^2)}^2 \right] <\infty,
	\end{align*} 
	with $ |\tilde{m}+e_3| = 1 $, a.e. in $ [0,T] \times \R^d $, $ \tilde{\P} $-a.s. 
	Thus, $ \tilde{M}:= \tilde{m}+e_3 $ is a solution of \eqref{eq: sLLG with Delta^2} in the sense of Definition \ref{Def: martingale solution}. 
	
	Using Kolmogorov's criterion, we deduce that $ \tilde{m} = \tilde{M}-e_3 $ has paths in $ \mathcal{C}^\sigma([0,T];\L^2) $ $ \tilde{\P} $-a.s. 
	Let $ 0\leq s \leq t \leq T $ and $ p \in [1,\infty) $, we have
	\begin{align*}
		\tilde{\E} \left[ \la \int_s^t |F(\tilde{M}(r))|_{\L^2} \dr \ra^{2p} \right]
		&\leq |t-s|^p \tilde{\E} \left[ \la \int_s^t |F(\tilde{M}(r))|_{\L^2}^2 \dr \ra^p \right] \\
		&\lesssim |t-s|^p \tilde{\E} \left[ \la \int_0^T \la |\nabla \tilde{m}|_{\H^1}^2 + |(\tilde{m}+e_3) \times \Delta^2 \tilde{m}|_{\L^2}^2 \ra(r) \dr \ra^p \right] \\
		&\lesssim |t-s|^p, \\
		\tilde{\E} \left[ \la \sum_k \int_s^t |S_k(\tilde{m}(r))|_{\L^2} \dr \ra^{2p} \right]
		&\leq \tilde{\E} \left[ \la \sum_k |g_k|_{W^{1,\infty}(\R^d;\R^d)}^2 \int_s^t |\nabla \tilde{m}(r)|_{\H^1} \dr \ra^{2p} \right] \\
		&\lesssim |t-s|^{2p} \tilde{\E} \left[ |\nabla \tilde{m}|_{L^\infty(0,T;\H^1)}^{2p} \right] \\
		&\lesssim |t-s|^{2p}, \\
		\tilde{\E} \left[ \la \sum_k \left| \int_s^t G_k(\tilde{m}(r)) \ \d\tilde{W}_k(r) \right|_{\L^2} \ra^{2p} \right]
		&= \tilde{\E} \left[ \la \sum_k \int_s^t |G_k(\tilde{m}(r))|^2_{\L^2} \dr \ra^p \right] \\
		&\leq \tilde{\E} \left[ \la \sum_k |g_k|_{L^\infty(\R^d;\R^d)}^2 \int_s^t |\tilde{m}(r)|_{\H^1}^2 \dr \ra^p \right]	\\
		&\lesssim |t-s|^p.
	\end{align*}
	Therefore,
	\begin{align*}
		\tilde{\E} \left[ |\tilde{m}(t) - \tilde{m}(s)|_{\L^2}^{2p} \right]
		&\lesssim |t-s|^p,
	\end{align*}  
	proving the $ \mathcal{C}^\sigma([0,T];\L^2) $-regularity for $ \sigma \in [0,\frac{1}{2}) $.

	\subsection{Pathwise uniqueness}\label{Section: pathwise unique}
	Let $ M_1 $ and $ M_2 $ be two martingale solutions of the equation \eqref{eq: sLLG with Delta^2} defined on the same filtered probability space and with the same Wiener process $ (\Omega, \mathcal{F}, (\mathcal{F}_t), \P, W) $. 
	As usual, we set $ m_1 := M_1 - e_3 $ and $ m_2 := M_2 - e_3 $. Then, $ m_1 $ and $ m_2 $ are martingale solutions of \eqref{eq: sLLG m}. 
	Note that in distributional sense
	\begin{align*}
		M_1 \times (M_1 \times \Delta M_1) 
		&= -|\nabla M_1|^2 M_1 - \Delta M_1, \\
		M_1 \times (M_1 \times \Delta^2 M_1) 
		&= \lb M_1, \Delta^2 M_1 \rb M_1 - \Delta^2 M_1.
	\end{align*}
	Let $ y:=M_1 - M_2 = m_1 - m_2 \in \L^2 $. 
	By the Gagliardo-Nirenberg's inequality, 
	\begin{equation}\label{eq: |Dy|L2, |Dy|L4, |y|Linf}
		\begin{aligned}
			|\nabla y|_{\L^2}^2 
			&= -\lb y, \Delta y \rb_{\L^2}
			\leq |y|_{\L^2} |\Delta y|_{\L^2} \\
			|\nabla y|_{\L^4}
			&\leq c |y|_{\L^2}^{\frac{1}{2}-\frac{d}{8}} |\Delta y|_{\L^2}^{\frac{1}{2}+\frac{d}{8}}, \\
			|y|_{\L^\infty} 
			&\leq c |y|_{\L^2}^{1-\frac{d}{4}} |\Delta y|_{\L^2}^{\frac{d}{4}}. 
		\end{aligned}
	\end{equation}
	In addition, $ |y|_{\L^\infty} \leq 2 $, $ y(0) = 0 $ and we have
	\begin{align*}
		\d y 
		&= 
		y \times (\Delta M_1 + \Delta^2 M_1) \dt 
		+ M_2 \times (\Delta y + \Delta^2 y) \dt - y \times he_3 \dt \\
		&\quad - \alpha \la \lb \nabla y, \nabla (M_1 + M_2) \rb M_1 + |\nabla M_2|^2 y \ra \dt \\
		&\quad + \alpha \la \lb y, \Delta^2 M_1 \rb M_1 + \lb M_2, \Delta^2 y \rb M_1 + \lb M_2, \Delta^2 M_2 \rb y \ra \dt \\
		&\quad - \alpha \la y \times \la M_1 \times he_3 \ra + M_2 \times (y \times he_3) \ra \dt \\
		&\quad - \gamma \la y \times \nabla_v M_1 + M_2 \times \nabla_v y \ra \dt \\
		&\quad - \alpha \la \Delta^2 y + \Delta y \ra \dt - \nabla_v y \dt 
		+ \frac{1}{2} \sum_k (\nabla_{g_k})^2 y \dt 
		- \sum_k \nabla_{g_k} y \dW_k.
	\end{align*}
	Applying It{\^o}'s lemma to $ \frac{1}{2}|y(t)|_{\L^2}^2 $, 
	\begin{align*}
		\frac{1}{2} \d|y(t)|_{\L^2}^2 
		&= 
		\lb y, M_2 \times (\Delta y + \Delta^2 y) - \alpha M_2 \times (y \times he_3) \rb_{\L^2} \dt \\
		&\quad - \alpha \lb y, \lb \nabla y, \nabla (M_1 + M_2) \rb M_1 \rb_{\L^2} \dt - \alpha \int_{\R^d} |\nabla M_2|^2 |y|^2 \dx \dt \\
		&\quad + \alpha \lb y, \lb y, \Delta^2 M_1 \rb M_1 + \lb M_2, \Delta^2 y \rb M_1 + \lb M_2, \Delta^2 M_2 \rb y \rb_{\L^2} \dt \\
		&\quad - \gamma \lb y, M_2 \times \nabla_v y \rb_{\L^2} \dt 
		+ \alpha \la |\nabla y|_{\L^2}^2 - |\Delta y|_{\L^2}^2 \ra \dt \\
		&\quad 
		+ \frac{1}{2} \sum_k \la \lb y, (\nabla_{g_k})^2 y \rb_{\L^2} + |\nabla_{g_k} y|_{\L^2}^2 \ra \dt - \sum_k \lb y, \nabla_{g_k} y \rb_{\L^2} \dW_k,
	\end{align*}
	where $ \langle y, \nabla_v y \rangle_{\L^2} = 0 $ and by \eqref{eq: |Dy|L2, |Dy|L4, |y|Linf}, 
	\begin{align*}
		|\nabla y|_{\L^2}^2 \leq \delta^{-1} |y|_{\L^2}^2 + \delta |\Delta y|_{\L^2}^2, 
	\end{align*}
	for $ \delta \in (0,1) $. 
	
	For the Stratonovich and It{\^o} correction terms,  
	\begin{align*}
		\sum_k \la \lb y, (\nabla_{g_k})^2 y \rb_{\L^2} + |\nabla_{g_k} y|_{\L^2}^2 \ra
		&= -\sum_k \lb y, ({\rm div} g_k) \nabla_{g_k} y \rb_{\L^2} \\
		&\lesssim \delta^{-1} \sum_k |g_k|_{W^{1,\infty}(\R^d;\R^d)}^2 |y|_{\L^2}^2 + \delta |\nabla y|_{\L^2}^2 \\
		&\lesssim |y|_{\L^2}^2 + \delta |\Delta y|_{\L^2}^2.
	\end{align*}
	Also, for the diffusion term, 
	\begin{align*}
		\sum_k \lb y, \nabla_{g_k} y \rb_{\L^2}^2 
		&= \frac{1}{2}\sum_k \lb y, ({\rm div} g_k) y \rb_{\L^2}^2 \\
		&\leq \frac{1}{2} \sum_k |{\rm div} g_k|_{\L^2}^2 |y|_{\L^\infty}^2 |y|_{\L^2}^2 \\
		&\leq 2\sum_k |g_k|_{H^1(\R^d;\R^d)}^2 |y|_{\L^2}^2, 
	\end{align*}
	where the right-hand side has finite expectation, implying that the It{\^o} integral is a well-defined continuous martingale. 
	
	We estimate the main drift part below.  
	\begin{align*}
		&\lb y, M_2 \times \Delta y - \alpha M_2 \times (y \times he_3) - \gamma M_2 \times \nabla_v y \rb_{\L^2} \\
		&= 
		\lb y, M_2 \times \Delta y \rb_{\L^2} 
		- \alpha \lb y, M_2 \times (y \times he_3) \rb_{\L^2} 
		- \gamma \lb y, M_2 \times \nabla_v y \rb_{\L^2} \\
		&\lesssim 
		|y|_{\L^2}^2 + \delta |\Delta y|_{\L^2}^2, \\
	\end{align*}
	and
	\begin{align*}
		\lb y, M_2 \times \Delta^2 y \rb_{\L^2} 
		&= 
		\lb y \times \Delta M_2, \Delta y \rb_{\L^2} 
		+ 2 \lb \nabla y \times \nabla M_2, \Delta y \rb_{\L^2} \\
		&\leq
		|y|_{\L^\infty} |\Delta M_2|_{\L^2} |\Delta y|_{\L^2} 
		+ 2 |\nabla y|_{\L^4} |\nabla M_2|_{\L^4} |\Delta y|_{\L^2} \\
		&\lesssim 
		|\Delta M_2|_{\L^2} |y|_{\L^2}^{1-\frac{d}{4}} |\Delta y|_{\L^2}^{1+\frac{d}{4}}
		+ |\nabla M_2|_{\H^1} |y|_{\L^2}^{\frac{1}{2}-\frac{d}{8}} |\Delta y|_{\L^2}^{\frac{3}{2} + \frac{d}{8}} \\
		&\lesssim	
		\la |\Delta M_2|_{\L^2}^{\frac{8}{4-d}} + |\Delta M_2|_{\L^2}^2 + |\nabla M_2|_{\H^2}^\frac{16}{4-d} \ra |y|_{\L^2}^2  
		+ \delta |\Delta y|_{\L^2}^2.
	\end{align*}
	Similarly, for the $ \alpha $ terms, 
	\begin{align*}
		\lb y, \lb \nabla y, \nabla (M_1 + M_2) \rb M_1 \rb_{\L^2} 
		&\leq |y|_{\L^\infty} |\nabla y|_{\L^2} |\nabla (M_1 + M_2) |_{\L^2} \\
		&\lesssim |y|_{\L^\infty} |y|_{\L^2}^\frac{1}{2} |\Delta y|_{\L^2}^\frac{1}{2} |\nabla (M_1 + M_2) |_{\L^2} \\ 
		&\lesssim |y|_{\L^2}^{\frac{3}{2}-\frac{d}{4}} |\Delta y|_{\L^2}^{\frac{1}{2}+ \frac{d}{4}} |\nabla (M_1 + M_2) |_{\L^2} \\
		&\lesssim |\nabla (M_1 + M_2)|_{\L^2}^{\frac{8}{6-d}}|y|_{\L^2}^2 + \delta|\Delta y|_{\L^2}^2,
	\end{align*}
	and
	\begin{align*}
		\lb y, \lb y, \Delta^2 M_1 \rb M_1 \rb_{\L^2} 
		&= \lb (y \cdot M_1) y, \Delta^2 M_1 \rb_{\L^2} \\
		&= \lb (\Delta y \cdot M_1)y + 2 (\nabla y \cdot \nabla M_1) y + (y \cdot M_1) \Delta y, \Delta M_1 \rb_{\L^2} \\
		&\quad + 2\lb (\nabla y \cdot M_1 + y \cdot \nabla M_1) \nabla y, \Delta M_1 \rb_{\L^2} 
		+ \int_{\R^d} \lb y, \Delta M_1 \rb^2 \dx \\
		&\leq 
		2\la |y|_{\L^\infty} |\Delta y|_{\L^2} + |\nabla y|_{\L^4}^2 \ra |\Delta M_1|_{\L^2} \\
		&\quad + 4 |y|_{\L^\infty} |\nabla y|_{\L^4} |\nabla M_1|_{\L^4} |\Delta M_1|_{\L^2} + |y|_{\L^\infty}^2 |\Delta M_1|_{\L^2}^2 \\
		&\lesssim 
		|y|_{\L^2}^{1-\frac{d}{4}} |\Delta y|_{\L^2}^{1+\frac{d}{4}} |\Delta M_1|_{\L^2} 
		+ \la |y|_{\L^2}^{\frac{3}{2}-\frac{3d}{8}} |\Delta y|_{\L^2}^{\frac{1}{2}+\frac{3d}{8}} + |y|_{\L^2}^{2-\frac{d}{2}} |\Delta y|_{\L^2}^\frac{d}{2} \ra |\nabla M_1|_{\H^1}^2 \\
		&\lesssim 
		\la |\Delta M_1|_{\L^2}^\frac{8}{4-d} + |\nabla M_1|_{\H^1}^\frac{32}{12-3d} \ra |y|_{\L^2}^2 
		+ \delta|\Delta y|_{\L^2}^2, \\
		\lb y, \lb M_2, \Delta^2 M_2 \rb y \rb_{\L^2} 
		&= \lb |y|^2 M_2, \Delta^2 M_2 \rb_{\L^2} \\
		&= \int_{\R^d} |y|^2 |\Delta M_2|^2 \ dx + 2\lb (y \cdot \Delta y + |\nabla y|^2) M_2 + 2(y \cdot \nabla y) \nabla M_2, \Delta M_2 \rb_{\L^2} \\
		&\leq 2\la |y|_{\L^\infty} |\Delta y|_{\L^2} + |\nabla y|_{\L^4}^2 + 2|y|_{\L^\infty} |\nabla y|_{\L^4} |\nabla M_2|_{\L^4} \ra |\Delta M_2|_{\L^2} + |y|_{\L^\infty}^2 |\Delta M_2|_{\L^2}^2 \\
		&\lesssim 
		|y|_{\L^2}^{1-\frac{d}{4}} |\Delta y|_{\L^2}^{1+\frac{d}{4}} |\Delta M_2|_{\L^2} 
		+ \la |y|_{\L^2}^{2-\frac{d}{2}} |\Delta y|_{\L^2}^\frac{d}{2} + |y|_{\L^2}^{\frac{3}{2}-\frac{3d}{8}} |\Delta y|_{\L^2}^{\frac{1}{2}+\frac{3d}{8}} \ra |\nabla M_2|_{\H^1} |\Delta M_2|_{\L^2} \\
		&\lesssim |y|_{\L^2}^2 \la |\Delta M_2|_{\L^2}^{\frac{8}{4-d}} + |\nabla M_2|_{\H^1}^{\frac{32}{12-3d}} \ra + \delta |\Delta y|_{\L^2}^2.	
	\end{align*}
	For the $ \Delta^2 y $ term, 
	\begin{align*}
		\lb y, \lb M_2, \Delta^2 y \rb M_1 \rb_{\L^2} 
		&= \lb (y \cdot M_1) M_2, \Delta^2 y \rb_{\L^2} \\
		&= \int_{\R^d} \lb \Delta y, M_1 \rb \lb \Delta y, M_2 \rb \dx \\
		&\quad + \lb (y \cdot \Delta M_1)M_2 + (y \cdot M_1) \Delta M_2 + 2 (y \cdot \nabla M_1) \nabla M_2, \Delta y \rb_{\L^2} \\
		&\quad + 2 \lb  (\nabla y \cdot \nabla M_1) M_2 + (\nabla y \cdot M_1) \nabla M_2, \Delta y \rb_{\L^2} \\
		&\leq 
		\int_{\R^d} \lb \Delta y, M_1 \rb \lb \Delta y, M_2 \rb \dx \\
		&\quad + |y|_{\L^\infty} |\Delta y|_{\L^2} \la |\Delta M_1|_{\L^2} + |\Delta M_2|_{\L^2} + |\nabla M_1|_{\L^4} |\nabla M_2|_{\L^4} \ra \\
		&\quad + 2 |\nabla y|_{\L^4} |\Delta y|_{\L^2} \la |\nabla M_1|_{\L^4} + |\nabla M_2|_{\L^4} \ra \\
		&\leq 
		\int_{\R^d} \lb \Delta y, M_1 \rb \lb \Delta y, M_2 \rb \dx \\
		&\quad + c |y|_{\L^2}^{1-\frac{d}{4}} |\Delta y|_{\L^2}^{1+\frac{d}{4}} \la |\Delta M_1|_{\L^2} + |\Delta M_2|_{\L^2} + |\nabla M_1|_{\H^1} |\nabla M_2|_{\H^1} \ra \\
		&\quad + c |y|_{\L^2}^{\frac{1}{2} - \frac{d}{8}} |\Delta y|_{\L^2}^{\frac{3}{2}+\frac{d}{8}} \la |\nabla M_1|_{\H^1} + |\nabla M_2|_{\H^1} \ra \\
		&\leq 
		\int_{\R^d} \lb \Delta y, M_1 \rb \lb \Delta y, M_2 \rb \dx + c \delta |\Delta y|_{\L^2}^2 \\
		&\quad + c(\delta^{-1}) \la |\Delta M_1|_{\L^2} + |\Delta M_2|_{\L^2} + |\nabla M_1|_{\H^1} |\nabla M_2|_{\H^1} \ra^{\frac{8}{4-d}} |y|_{\L^2}^2 \\
		&\quad + c(\delta^{-1}) \la |\nabla M_1|_{\H^1} + |\nabla M_2|_{\H^1} \ra^\frac{16}{4-d} |y|_{\L^2}^2, 
	\end{align*}
	where
	\begin{align*}
		\int_{\R^d} \lb \Delta y, M_1 \rb \lb \Delta y, M_2 \rb \dx 
		&= \frac{1}{2} \int_{\R^d} \lb \Delta y, M_1 \rb \la |\nabla M_2|^2 + \lb M_2, \Delta M_1 \rb \ra \ds \\
		&\quad - \frac{1}{2} \int_{\R^d} \lb \Delta y, M_2 \rb \la |\nabla M_1|^2 + \lb M_1, \Delta M_2 \rb \ra \dx \\
		&= 
		\frac{1}{2} \int_{\R^d} \lb \Delta y, y \rb \la |\nabla M_2|^2 + \lb M_2, \Delta M_1 \rb \ra \dx \\
		&\quad + \frac{1}{2} \int_{\R^d} \lb \Delta y, M_2 \rb \la \lb \nabla (M_1 + M_2), -\nabla y \rb + \lb M_2, \Delta y \rb - \lb y, \Delta M_2 \rb \ra \dx \\
		&\leq 
		\frac{1}{2} |\Delta y|_{\L^2} |y|_{\L^\infty} \la |\nabla M_2|_{\L^4}^2 + |\Delta M_1|_{\L^2} + |\Delta M_2|_{\L^2} \ra \\
		&\quad + \frac{1}{2} |\Delta y|_{\L^2} |\nabla y|_{\L^4} |\nabla (M_1 + M_2)|_{\L^4} + \frac{1}{2}|\Delta y|_{\L^2}^2 \\
		&\leq
		\frac{1}{2}|\Delta y|_{\L^2}^2 
		+ c |y|_{\L^2}^{1-\frac{d}{4}}|\Delta y|_{\L^2}^{1+\frac{d}{4}} \la |\nabla M_2|_{\H^1}^2 + |\Delta M_1|_{\L^2} + |\Delta M_2|_{\L^2} \ra \\
		&\quad + c |y|_{\L^2}^{\frac{1}{2} - \frac{d}{8}} |\Delta y|_{\L^2}^{\frac{3}{2}+\frac{d}{8}}|\nabla (M_1 + M_2)|_{\H^1} \\
		&\leq
		\frac{1}{2}|\Delta y|_{\L^2}^2 
		+ c\delta |\Delta y|_{\L^2}^2 \\
		&\quad + c(\delta^{-1}) \la \la |\nabla M_2|_{\H^1}^2 + |\Delta M_1|_{\L^2} + |\Delta M_2|_{\L^2} \ra^\frac{8}{4-d} + |\nabla (M_1 + M_2)|_{\H^1}^\frac{16}{4-d} \ra |y|_{\L^2}^2. 
	\end{align*}
	
	Thus, 
	\begin{align*}
		\frac{1}{2}\d |y(t)|_{\L^2}^2 \leq \phi(t) |y(t)|_{\L^2}^2 \dt + \alpha \la c_1 \delta-\frac{1}{2} \ra |\Delta y(t)|_{\L^2}^2 \dt - \sum_k \lb y, \nabla_{g_k} y \rb_{\L^2} \dW_k,
	\end{align*}
	for some constant $ c_1 $ and
	\begin{align*}
		\phi(t) = c(h,\delta^{-1}) \la 1+|\nabla M_1|^{q}_{\H^1} + |\nabla M_2|^{q}_{\H^1} \ra, 
	\end{align*}
	for some $ q \in [2,\infty) $ depending on $ d $. 
	Since $ \nabla M_1, \nabla M_2 $ in $ L^{2p}(\Omega; L^\infty(0,T;\H^1)) $, we have $ \int_0^T \phi(t) \dt < \infty $, $ \P $-a.s. 
	Let $ \delta $ be sufficiently small such that $ c_1 \delta-\frac{1}{2} < 0 $. 
	
	If $ {\rm div} g_k = 0 $ for all $ k $, then applying Gronwall's inequality directly and using the fact $ y(0) = 0 $,
	\begin{align*}
		|y(t)|_{\L^2}^2 \leq |y(0)|_{\L^2}^2 e^{\int_0^T \phi(t) \dt} = 0, \quad \P\text{-a.s.} 
	\end{align*}
	Otherwise, let $ X(t) := \frac{1}{2}|y(t)|_{\L^2}^2 e^{-2 \int_0^t \phi(s) \ds} $. 
	Then we have
	\begin{align*}
		\d X(t) 
		&= \frac{1}{2}\lb \d |y(t)|_{\L^2}^2, e^{-2\int_0^t \phi(s) \ds} \rb 
		+ \frac{1}{2} \lb |y(t)|_{\L^2}^2, -2 \phi(t) e^{-2\int_0^t \phi(s) \ds} \rb \\
		&\quad + \frac{1}{2} \lb \d |y(t)|_{\L^2}^2, \d e^{-2\int_0^t \phi(s) \ds} \rb \\
		&\leq -e^{-2\int_0^t \phi(s) \ds} \sum_k \lb y, \nabla_{g_k} y \rb \dW_k,
	\end{align*}
	where the process $ M(t):= \sum_k \int_0^t e^{-2\int_0^t \phi(s) \ds} \langle y, \nabla_{g_k} y \dW_k \rangle_{\L^2} $ is a martingale. 
	Hence, $ \E[X(t)] \leq \E[M(t)] = 0 $, implying that
	\begin{align*}
		\E \left[ |y(t)|_{\L^2}^2 \right] = 0, \quad t \in [0,T].  
	\end{align*}	
	This proves the pathwise uniqueness of the solution of \eqref{eq: sLLG with Delta^2}, concluding the proof of Theorem \ref{Theorem: E!}.

\appendix

\section{Preliminary estimates}\label{Section: preliminary estimates}
	
	\subsection{Useful formulae}\label{Section: useful formulae}
	Let $ u,w \in \H^2 $ and $ f, g \in H^1(\R^d; \R^d) $, $ d=2,3 $. 
	Then
	\begin{align*}
		&(\nabla_g)(u \times w) = \nabla_g u \times w + u \times \nabla_g w, \\
		&(\nabla_g)^2 u = \sum_{i=1}^d g_i^2 \partial_i^2 u + \sum_{i,j=1,i \neq j}^d g_i g_j \partial_{ij} u + \nabla_{\nabla_g g} u, \\
		&\nabla_f (\nabla_g u) = \sum_{i,j=1}^d f_i g_j \partial_{ij} u + \nabla_{\nabla_f g} u.
	\end{align*}
	For $ u,w $ with suitable decay properties, 
	\begin{align}
		\lb \nabla_f u, w \rb_{\L^2} 
		&= -\lb u, \nabla_f w + ({\rm div} f) w \rb_{\L^2}, \label{eq: <Dfu,w>}\\
		\lb \nabla_f u, \nabla_g w \rb_{\L^2} 
		&= - \sum_{i,j=1}^d \lb f_ig_j \partial_{ij}u + \partial_j(f_ig_j) \partial_i u, w \rb_{\L^2}. \label{eq: <Dfu,Dgw>}
	\end{align}
	
	By \eqref{eq: <Dfu,w>}, for $ f \in W^{2,\infty}(\R^d;\R^d) $, 
	\begin{equation}\label{eq: <u, Dfu> simplify}
		\lb u, \nabla_f u \rb_{\L^2}
		= -\frac{1}{2} \lb ({\rm div} f) u, u \rb_{\L^2},
	\end{equation}
	and
	\begin{align*}
		-\lb |u|^2 u, \nabla_f u \rb_{\L^2}
		&= \lb \nabla_f (|u|^2 u) + ({\rm div} f) |u|^2 u, u \rb_{\L^2} \\
		&= 3\lb |u|^2 \nabla_f u, u \rb_{\L^2} 
		+ \lb ({\rm div} f) |u|^2 u, u \rb_{\L^2} \\
		&= \frac{1}{4} \lb ({\rm div} f) |u|^2 u, u \rb_{\L^2},
	\end{align*}
	which imply 
	\begin{equation}\label{eq: <(1-|u|^2)u,Dfu>}
		\begin{aligned}
			\lb (1-|u|^2)u, \nabla_f u \rb_{\L^2} 
			&= \frac{1}{4} \int_{\R^d} ({\rm div}f) \la |u|^4 -2|u|^2 \ra \dx \\
			&= \frac{1}{4} \int_{\R^d} ({\rm div}f) (1-|u|^2)^2 \dx - \frac{1}{4} \int_{\R^d} {\rm div} f \dx,
		\end{aligned}
	\end{equation}
	where $ \int_{\R^d} {\rm div} f \dx = 0 $ if $ f $ vanishes at infinity. 

	Without loss of generality, assume that $ j $ is supported on the (open) unit ball. 
	Let $ B_\epsilon(x) $ denote the open ball of radius $ \epsilon>0 $ centred at $ x \in \R^d $. 
	Then for weighted $ \L^2 $-norm, 
	\begin{align*}
		|J_\epsilon u|_{\L^2_\rho}^2
		&= \int_{\R^d} \left| \int_{B_\epsilon(x)} j_\epsilon(x-y) u(y) \dy \right|^2 \rho(x) \dx \\
		&\leq \int_{\R^d} \la \int_{B_\epsilon(x)} j_\epsilon(x-y) \rho^{-1}(y) \dy \ra \la \int_{\R^d} j_\epsilon(x-y) |u(y)|^2 \rho(y) \dy \ra \rho(x) \dx,
	\end{align*}
	where $ \rho^{-1}(y) = (1+|y|^2)^2 $ is bounded from above by $ (1+(|x|+\epsilon)^2)^2 $ in $ B_\epsilon(x) $, implying
	\begin{align*}
		\int_{B_\epsilon(x)} j_\epsilon(x-y) \rho^{-1}(y) \rho(x) \dy 
		&\leq (1+(|x|+\epsilon)^2)^2 \rho(x) \int_{B_\epsilon(x)} j_\epsilon(x-y) \dy \\
		&= (1+(|x|+\epsilon)^2)^2 \rho(x) \\
		&\leq c, \quad \forall x \in \R^d. 
	\end{align*}
	Thus, 
	\begin{equation}\label{eq: |Ju|_L2rho <= c|u|_L2rho}
		\begin{aligned}
			|J_\epsilon u|_{\L^2_\rho}^2
			&\leq c \int_{\R^d} \int_{\R^d} j_\epsilon(x-y) |u(y)|^2 \rho(y) \dy \dx \\
			&= c \int_{\R^d} \la \int_{\R^d} j_\epsilon(x-y) \dx \ra |u(y)|^2 \rho(y) \dy \\
			&= c |u|_{\L^2_\rho}^2. 
		\end{aligned}
	\end{equation}
	Similarly,
	\begin{equation}\label{eq: <Ju,w>_L2rho} 
		\begin{aligned}
			\lb J_\epsilon u, w \rb_{\L^2_\rho}
			&= \lb u, J_\epsilon(w\rho) \rb_{\L^2} \\
			&= \lb u, J_\epsilon w \rb_{\L^2_\rho} 
			+ \lb u, J_\epsilon(w\rho) - w \rho \rb_{\L^2} 
			+ \lb u, w - J_\epsilon w \rb_{\L^2_\rho} \\
			&\leq 
			\lb u, J_\epsilon w \rb_{\L^2_\rho} 
			+ |u|_{\L^2} |J_\epsilon(w\rho)-w \rho|_{\L^2} + |u|_{\L^2_\rho} |J_\epsilon w - w |_{\L^2},
		\end{aligned}
	\end{equation}
	where the last two terms converge to $ 0 $ as $ \epsilon \to 0 $ for $ u,w \in \L^2 $.

\subsection{Estimates for mixed partial derivatives}\label{Section: mixed partial derivatives}	
	\begin{lemma}\label{Lemma: D4 vs D1 estimates}
		Let $ u:\R^d \to \R^3 $ be a sufficiently smooth function with suitable decay properties. 
		Then
		\begin{equation}\label{eq: Dij vs Delta}	
			\sum_{i=1}^d |\partial_i^2 u|_{\L^2}^2 
			\leq |\Delta u|_{\L^2}^2, \quad 
			\sum_{i,j=1}^d |\partial_{ij} u|_{\L^2}^2 
			\leq |\Delta u|^2_{\L^2}.
		\end{equation}
		Let $ f \in W^{2,\infty}(\R^d;\R) $. Then for $ i,j \in \{1,\ldots,d\} $, 
		\begin{align}
			\lb \partial_i^4 u, f \partial_j u \rb_{\L^2} 
			&\lesssim |\partial_j u|^2_{\L^2} + |\partial_i^2 u|^2_{\L^2} + |\partial_{ij}u|^2_{\L^2}, \label{eq: Di4-Dj}\\
			\lb \partial_i^2 \partial_j^2 u, f \partial_i u \rb_{\L^2} 
			&\lesssim |\partial_i u|^2_{\L^2} + |\partial_i^2 u|^2_{\L^2} + |\partial_j^2 u|^2_{\L^2} + |\partial_{ij}u|^2_{\L^2}, \label{eq: Di2Dj2-Di} \\
			\lb \partial_i^3 \partial_j u, f \partial_i u \rb_{\L^2} 
			&\lesssim |\partial_i u|^2_{\L^2} + |\partial_i^2 u|^2_{\L^2} + |\partial_{ij}u|^2_{\L^2}, \label{eq: Di3Dj-Di} \\
			\lb \partial_i^3 \partial_j u, f \partial_j u \rb_{\L^2} 
			&\lesssim |\partial_j u|^2_{\L^2} + |\partial_i^2 u|^2_{\L^2} + |\partial_j^2 u|^2_{\L^2} + |\partial_{ij}u|^2_{\L^2}. \label{eq: Di3Dj-Dj}
		\end{align} 
	\end{lemma}
	\begin{proof}
		For the inequalities in \eqref{eq: Dij vs Delta}, we first observe that
		\begin{align*}
			|\partial_{ij} u|_{\L^2}^2 
			= \lb \partial_i^2 u, \partial_j^2 u \rb_{\L^2}.
		\end{align*}
		Then we have
		\begin{align*}
			\sum_{i=1}^d |\partial_i^2 u|_{\L^2}^2 
			= |\Delta u|^2_{\L^2} - 2 \sum_{i,j=1, i\neq j}^d \lb \partial_i^2 u, \partial_j^2 u \rb_{\L^2} 
			= |\Delta u|^2_{\L^2} - 2 \sum_{i,j=1, i \neq j}^d |\partial_{ij}u|_{\L^2}^2 
			\leq |\Delta u|^2_{\L^2},
		\end{align*}
		and as a result,
		\begin{align*}
			\sum_{i,j=1}^d |\partial_{ij} u|_{\L^2}^2  
			\leq \frac{1}{2} \la \sum_{i=1}^d |\partial_i^2 u|_{\L^2}^2 + \sum_{j=1}^d |\partial_j^2 u|_{\L^2}^2 \ra
			\leq |\Delta u|_{\L^2}^2.
		\end{align*}
		
		In order to deduce the inequalities \eqref{eq: Di4-Dj}-\eqref{eq: Di3Dj-Dj}, 
		we first observe that
		\begin{equation}\label{eq: Di2-Di2Dj}
			\lb \partial_i^2 u, f \partial_i^2 \partial_j u \rb_{\L^2} 
			= - \lb \partial_i^2 u, \partial_j f \ \partial_i^2 u + f \partial_i^2 \partial_j u \rb_{\L^2} 
			= - \frac{1}{2} \lb \partial_i^2 u, \partial_j f \ \partial_i^2 u \rb_{\L^2}.
		\end{equation}
		Then for \eqref{eq: Di4-Dj}, 
		\begin{align*}
			\lb \partial_i^4 u, f \partial_j u \rb_{\L^2} 
			&= \lb \partial_i^2 u , \partial_i^2 f \ \partial_j u + f \partial_i^2 \partial_j u + 2 \partial_i f \ \partial_{ij} u \rb_{\L^2} \\
			&= \lb \partial_i^2 u, \partial_i^2 f \ \partial_j u + 2 \partial_i f \ \partial_{ij} u \rb_{\L^2} - \frac{1}{2} \lb \partial_i^2 u, \partial_j f \ \partial_i^2 u \rb_{\L^2} \\
			&\lesssim 
			|\partial_j u|^2_{\L^2} + |\partial_i^2 u|^2_{\L^2} + |\partial_{ij}u|^2_{\L^2},
		\end{align*} 
		where the second equality holds by \eqref{eq: Di2-Di2Dj}. 
		Similarly, for \eqref{eq: Di2Dj2-Di},
		\begin{align*}
			\lb \partial_i^2 \partial_j^2 u, f \partial_i u \rb_{\L^2} 
			&= \lb \partial_j^2 u, \partial_i^2 f \ \partial_i u + f \partial_i^3 u + 2 \partial_i f \ \partial_i^2 u \rb_{\L^2}  \\
			&= \lb \partial_j^2 u, \partial_i^2 f \ \partial_i u + 2 \partial_i f \ \partial_i^2 u \rb_{\L^2} - \lb \partial_i f \ \partial_j^2 u, \partial_i^2 u \rb_{\L^2} - \lb f \ \partial_i \partial_j^2 u, \partial_i^2 u \rb_{\L^2}  \\
			&= \lb \partial_j^2 u, \partial_i^2 f \ \partial_i u + \partial_i f \ \partial_i^2 u \rb_{\L^2} - \lb \partial_i u, \partial_j^2 \la f \partial_i^2 u \ra \rb_{\L^2} \\
			&= \lb \partial_j^2 u, \partial_i^2 f \ \partial_i u + \partial_i f \ \partial_i^2 u \rb_{\L^2} - \lb \partial_i u, \partial_j^2 f \ \partial_i^2 u \rb_{\L^2} - 2 \lb \partial_i u, \partial_j f \ \partial_j \partial_i^2 u \rb_{\L^2} \\
			&\quad - \lb f \partial_i u, \partial_i^2 \partial_j^2 u \rb_{\L^2} \\
			&= \frac{1}{2} \lb \partial_j^2 u, \partial_i^2 f \ \partial_i u + \partial_i f \ \partial_i^2 u \rb_{\L^2} - \frac{1}{2} \lb \partial_i u, \partial_j^2 f \ \partial_i^2 u \rb_{\L^2} 
			- \lb \partial_i u, \partial_j f \ \partial_j \partial_i^2 u \rb_{\L^2} \\
			&= \frac{1}{2} \lb \partial_i u, \partial_i^2 f \ \partial_j^2 u - \partial_j^2 f \ \partial_i^2 u \rb_{\L^2} 
			+ \frac{1}{2} \lb \partial_j^2 u, \partial_i f \ \partial_i^2 u \rb_{\L^2} 	
			+ \lb \partial_j^2 f \ \partial_i u + \partial_j f \ \partial_{ij}u, \partial_i^2 u \rb_{\L^2}\\
			&\lesssim |\partial_i u|^2_{\L^2} + |\partial_i^2 u|^2_{\L^2} + |\partial_j^2 u|^2_{\L^2} + |\partial_{ij}u|^2_{\L^2}, 
		\end{align*}
		where the fifth equality holds since the last term in the fourth equality is the negative counterpart of the left-hand side. 
		For \eqref{eq: Di3Dj-Di},  
		\begin{align*}
			\lb \partial_i^3 \partial_j u, f \partial_i u \rb_{\L^2} 
			&= \lb \partial_{ij}u, \partial_i^2 (f \partial_i u) \rb_{\L^2} \\
			&= \lb \partial_{ij}u, \partial_i^2 f \ \partial_i u + f \ \partial_i^3 u + 2 \partial_i f \ \partial_i^2 u \rb_{\L^2} \\
			&= \lb \partial_{ij}u, \partial_i^2 f \ \partial_i u + 2 \partial_i f \ \partial_i^2 u \rb_{\L^2} 
			- \lb \partial_i u, \partial_j f \ \partial_i^3 u \rb_{\L^2} - \lb \partial_i u, f \partial_i^3 \partial_j u \rb_{\L^2} \\
			&= \frac{1}{2} \lb \partial_{ij}u, \partial_i^2 f \ \partial_i u + 2 \partial_i f \ \partial_i^2 u \rb_{\L^2} 
			- \frac{1}{2} \lb \partial_i u, \partial_j f \ \partial_i^3 u \rb_{\L^2} \\
			&= \frac{1}{2} \lb \partial_{ij}u, \partial_i^2 f \ \partial_i u \rb_{\L^2} 
			+ \lb \partial_{ij}u, \partial_i f \ \partial_i^2 u \rb_{\L^2} 
			+ \frac{1}{2} \lb \partial_j f \ \partial_i^2 u + \partial_{ij}f \ \partial_i u, \partial_i^2 u \rb_{\L^2} \\
			&\lesssim |\partial_iu|^2_{\L^2} + |\partial_i^2 u|^2_{\L^2} + |\partial_{ij}u|^2_{\L^2}.
		\end{align*}
		For \eqref{eq: Di3Dj-Dj}, the derivation is similar to that of \eqref{eq: Di3Dj-Di} and we obtain
		\begin{align*}
			\lb \partial_i^3 \partial_j u, f \partial_j u \rb_{\L^2}
			&= \frac{1}{2} \lb \partial_{ij}u, \partial_i^2 f \ \partial_j u \rb_{\L^2} 
			+ \lb \partial_{ij}u, \partial_i f \ \partial_{ij}u \rb_{\L^2} 
			+ \frac{1}{2} \lb \partial_i f \ \partial_j^2 u + \partial_{ij}f \ \partial_j u, \partial_i^2u \rb_{\L^2} \\
			&\lesssim |\partial_j u|^2_{\L^2} + |\partial_i^2 u|^2_{\L^2} + |\partial_j^2 u|^2 + |\partial_{ij}u|^2_{\L^2},
		\end{align*} 
		concluding the proof.
	\end{proof}
	
	This lemma allows us to treat a higher-order term $ \langle \Delta^2 (\nabla_f u), \nabla_f u \rangle $. 
	Note that 
	\begin{align*}
		\Delta^2 (\nabla_f u) 
		&= \nabla_f (\Delta^2 u) 
		+ 4 \sum_{i,j=1}^d \partial_i f_j \partial_{ij} (\Delta u) 
		+ 2\la \nabla_{\Delta f} (\Delta u) + 2 \sum_{i,j,k=1}^d \partial_{ij} f_k \partial_{ijk}u \ra \\
		&\quad + \la 4 \sum_{i,j=1}^d \partial_i(\Delta f_j) \partial_{ij}u + \nabla_{\Delta^2 f}u \ra \\
		&= \nabla_f (\Delta^2 u) + 4{\rm T}_{1a}(u) + 2{\rm T}_{1b}(u) + {\rm T}_{1c}(u).
	\end{align*}
	For $ f \in H^4(\R^d;\R^d) $, using Lemma \ref{Lemma: D4 vs D1 estimates} we have
	\begin{equation}\label{eq: T1a, T1b, T1c, <D4u,D1fu>}
		\begin{aligned}
			\lb {\rm T}_{1a}(u), \nabla_f u \rb_{\L^2}	
			&= \sum_{i,j} \lb \partial_i f_j \partial_{ij} (\Delta u), \nabla_f u \rb_{\L^2} \\
			&\lesssim |f|^2_{W^{1,\infty}(\R^d;\R^d)} |\nabla u|^2_{\H^1}, \\
			\lb {\rm T}_{1b}(u), \nabla_f u \rb_{\L^2}
			&= \lb \nabla_{\Delta f}(\Delta u), \nabla_f u \rb_{\L^2} + 2 \sum_{i,j,k} \lb \partial_{ij}f_k \ \partial_{ijk}u, \nabla_f u \rb_{\L^2} \\
			&= -\lb \Delta u, \nabla_{\Delta f} \nabla_f u + ({\rm div} \Delta f) \nabla_f u \rb_{\L^2} 
			- 2 \sum_{i,j,k} \lb \partial_{ij}u, \partial_k \la \partial_{ij}f_k \nabla_f u \ra \rb_{\L^2} \\
			&\lesssim |f|^2_{W^{2,\infty} \cap W^{3,4}(\R^d;\R^d)} |\nabla u|^2_{\H^1}, \\
			\lb {\rm T}_{1c}(u), \nabla_f u \rb_{\L^2} 
			&= 4\sum_{i,j} \lb \partial_i(\Delta f_j) \partial_{ij}u, \nabla_f u \rb_{\L^2} + \lb \nabla_{\Delta^2f} u, \nabla_f u \rb_{\L^2} \\
			&\lesssim |f|^2_{W^{2,\infty} \cap W^{3,4} \cap H^4(\R^d;\R^d)} |\nabla u|^2_{\H^1}, \\
			\lb \Delta^2 u, \nabla_f u \rb_{\L^2} 
			&\lesssim |f|^2_{W^{2,\infty}(\R^d;\R^d)} |\nabla u|^2_{\H^1},
		\end{aligned}
	\end{equation}
	where $ |f|_{W^{2,\infty}(\R^d;\R^d)}+|f|_{W^{3,4}(\R^d;\R^d)} \lesssim |f|_{H^4(\R^d;\R^d)} $ for $ d \leq 3 $. 
	The estimates \eqref{eq: T1a, T1b, T1c, <D4u,D1fu>} help us to address It{\^o} corrections in Step 2(ii) of the proof of Lemma \ref{Lemma: meps H2-est}.

\section{On the torus}\label{Section: torus}
	For equation \eqref{eq: sLLG with Delta^2} formulated on the torus $ \T^d \subset \R^d $, Theorem \ref{Theorem: E!} still holds with $ \T^d $ in place of $ \R^d $.   
	Since the domain is bounded, the proof can be simplified in this case. 
	In this section, let $ \L^p $ and $ \H^\sigma $ denote the usual Lebesgue and Hilbert spaces $ L^p(\T^d;\R^3) $ and $ W^{\sigma,2}(\T^d;\R^3) $, respectively. 
	We can work with $ \m $ directly and only need Faedo-Galerkin approximations $ \{ \m_n \} $ with a cut-off function $ \psi_R $ that depends on $ |\m(t)|_{\L^\infty} $ instead of $ |\m(t,x)|^2 $. More explicitly, 
	let $ \{{\rm \bf e}_i\} $ be an orthonormal basis of $ \L^2 $ consisting of eigenvectors of the negative Laplacian $ -\Delta $ in $ \T^d $ and let $ \H_n := \text{span} \{ {\rm \bf e}_i, \ldots, {\rm \bf e}_n \} $. 
	Let $ \Pi_n $ be the orthogonal projection from $ \L^2 $ onto $ \H_n $, and we formulate the approximating equation:
	\begin{equation}\label{eq: dm-n-R}
		\d \m_n(t)
		= F^R_n(\m_n(t)) \dt + \frac{1}{2} \sum_{k=1}^\infty S_{k,n}(\m_n(t)) \dt + \sum_{k=1}^\infty G_{k,n}(\m_n(t)) \dW_k(t), 
		\quad
		m_n(0) =  \Pi_n m_0, 
	\end{equation}
	where 
	\begin{align*}
		F^R_n : \H_n \ni u &\mapsto \psi_R(|u|_{\L^\infty}) \Pi_n \bar{F}(u) - \Pi_n \nabla_v u \in \H_n, \\ 
		S_{k,n}: \H_n \ni u &\mapsto \Pi_n S_k(u) \in \H_n, \\ 
		G_{k,n} : \H_n \ni u &\mapsto \Pi_n G_k(u) \in \H_n,  
	\end{align*}
	are locally Lipschitz. 
	The uniform estimates can be obtained as in Lemma \ref{Lemma: meps H2-est}, and tightness and compact embeddings hold for the usual Lebesgue and Sobolev spaces without weight or localisation. Then, for the new processes $ \tilde{M}_n $ and $ \tilde{M} $ (deduced from Skorohod theorem) with strong convergence in $ L^{2p}(\tilde{\Omega}; L^2(0,T;\L^2)) $, convergence of the cut-off function $ \tilde{\psi}_n^R = \psi_R(|\tilde{\m}_n|_{\L^\infty}) \to \tilde{\psi}^R = \psi_R(|\tilde{\m}|_{\L^\infty}) $ in $ L^{2p}(\tilde{\Omega}; L^2(0,T)) $ is immediate from Gagliardo-Nirenberg inequality: 
	\begin{align*}
		\left| \tilde{\psi}_n^R(t) - \tilde{\psi}^R(t) \right| 
		&\leq \sup_{y \in \R} |\psi_R'(y)| \ \left| | \tilde{M}_n(t)|_{\L^\infty}-|\tilde{M}(t)|_{\L^\infty} \right| \\
		&\lesssim |\tilde{M}_n(t)- \tilde{M}(t)|_{\L^\infty} \\
		&\lesssim |\tilde{M}_n(t)- \tilde{M}(t)|_{\L^2}^{1-\frac{d}{4}} |\tilde{M}_n(t)- \tilde{M}(t)|_{\H^2}^\frac{d}{4}.
	\end{align*}
	Thus, we obtain similar convergences to Lemma \ref{Lemma: conv Feps,Seps,Geps} in $ \L^2 = L^2(\T^d;\R^3) $ without knowing $ \tilde{\psi}^R(t) = 1 $, and then we can deduce the existence of weak martingale solution to 
	\begin{equation}\label{eq: sLLG-psi}
		\tilde{\m}(t)
		= \tilde{\m}_0 + \int_0^t \la \tilde{\psi}^R \bar{F}(\tilde{\m}) - \nabla_v \tilde{\m} \ra (s) \ds + \frac{1}{2} \sum_k \int_0^t S_k(\tilde{\m})(s) \ds + \sum_k \int_0^t G_k(\tilde{\m})(s) \ \d \tilde{W}_k(s).
	\end{equation}
	Finally, $ |\tilde{\m}(t,x)| =1 $ (and thus $ \tilde{\psi}^R(t) = 1 $ for $ R>1 $) for a.e.-$ (t,x) $, $ \tilde{\P} $-a.s. can be shown using a similar but simpler calculation to Lemma \ref{Lemma: |meps+e3| <= remainder}, and the pathwise uniqueness follows as in Section \ref{Section: pathwise unique}.  

\end{document}